\documentclass[11pt]{amsart}
\usepackage{enumitem, amssymb}
\usepackage{todonotes}
\usepackage{fontawesome}
\usepackage{hyperref}
\usepackage{tikz}
\numberwithin{equation}{section}

\DeclareMathOperator{\sr}{sr}

\DeclareMathOperator{\OCAT}{OCA_T}
\DeclareMathOperator{\MA}{MA}

\DeclareMathOperator{\dist}{dist} 
\DeclareMathOperator{\Fin}{Fin}

\newcounter{my_enumerate_counter}

\DeclareMathOperator{\coTh}{coTh}
\DeclareMathOperator{\Th}{Th}

\newcommand{\cP}{\mathcal P}

\newcommand{\xproj}{x\text{ projection}}
\newcommand{\zproj}{z\text{ projection}}

\newcommand{\bbC}{{\mathbb C}}

\newcommand{\bbN}{{\mathbb N}}

\newcommand{\bbI}{{\mathbb I}}

\newcommand{\calL}{{\mathcal L}}

\newcommand{\cU}{{\mathcal U}}

\DeclareMathOperator{\Sp}{sp}

\newcommand{\cstar}{$\mathrm{C}^*$}
\newcommand{\cst}{\mathrm{C}^*}

\newcommand{\bfF}{\mathsf F}

\DeclareMathOperator{\id}{id}

\newcommand{\bt}{\mathsf t}

\newcommand{\bbF}{{\mathbb F}}

\newcommand{\bbS}{{\mathbb S}}

\newcommand{\bbQ}{\mathbb Q}
\newcommand{\bbR}{\mathbb R}
\newcommand{\bbH}{\mathbb H}

\newcommand{\cI}{{\mathcal I}}

\newcommand{\rs}{\restriction}

\newcommand{\cC}{\mathcal C}
\newcommand{\cF}{\mathcal F}

\newcommand{\cB}{\mathcal B}

\newcommand{\cM}{\mathcal M}

\newcommand{\cPN}{\cP(\bbN)}

\DeclareMathOperator{\Clop}{Clop}

\DeclareMathOperator{\Sep}{Sep} 

\newcommand{\dminus}{\dot -}

\newtheorem{theorem}{Theorem}[section]
\newtheorem*{theorem*}{Theorem}
\newtheorem{proposition}[theorem]{Proposition}

\newtheorem*{proposition*}{Proposition}
\newtheorem{lemma}[theorem]{Lemma}
\newtheorem*{lemma*}{Lemma}
\newtheorem{corollary}[theorem]{Corollary}
\newtheorem*{corollary*}{Corollar}

\newtheorem*{fact*}{Fact}
\theoremstyle{definition}
\newtheorem{definition}[theorem]{Definition}
\newtheorem*{definition*}{Definition}

\newtheorem*{claim*}{Claim}

\newtheorem*{conjecture*}{Conjecture}

\theoremstyle{remark}
\newtheorem{example}[theorem]{Example}
\newtheorem*{example*}{Example}
\newtheorem{remark}[theorem]{Remark}
\newtheorem*{remark*}{Remark}

\newtheorem*{note*}{Note}
\newtheorem*{question*}{Question}

\newtheorem*{funfact}{Fun fact}

\author[Ilijas Farah]{Ilijas Farah}
\thanks{Partially supported by NSERC.}
\thanks{Many thanks to Jan van Mill for a conversation that resulted in this paper.}

\address{Department of Mathematics and Statistics,
York University,
4700 Keele Street,
Toronto, Ontario, Canada, M3J
1P3} 
\address{Matemati\v cki Institut SANU,
Kneza Mihaila 36,
11\,000 Beograd, p.p. 367,
Serbia}
\email{email: ifarah@yorku.ca}
\urladdr{hhttps://ifarah.mathstats.yorku.ca}

\thanks{ORCID iD https://orcid.org/0000-0001-7703-6931}

\title{Applications of the Gelfand--Naimark duality}
\date{\today}

\begin{document}

\begin{abstract}
Stone duality is an indispensable tool for the study of compact, zero-dimensional, Hausdorff spaces. 
In the case of general compact Hausdorff spaces one can get quite a bit of mileage by considering the `Wallman duality' between compact spaces and lattices of closed sets. I will argue that the Gelfand--Naimark duality between compact Hausdorff spaces and unital, commutative \cstar-algebras provides great insight into compact Hausdorff spaces as well as into \v Cech--Stone remainders and their autohomeomorphisms in particular. 
\end{abstract}
\maketitle
\tableofcontents

Let me start by stating that there is nothing new here. Perhaps it is a mere coincidence that a 1960 classic on (among other things) \v Cech--Stone compactifications is called `Rings of continuous functions' (\cite{gillmanrings}), while the series of four papers some two decades prior by Murray and von Neumann introduced operator algebras (now known as von Neumann algebras) under the name of `Rings of operators' (\cite{murray1943rings} and references thereof). 
Some readers may object that every theorem about compact Hausdorff spaces proved by using the Gelfand--Naimark duality can be also proved by using Wallman duality. Yes, but all of those theorems can also be proved from scratch, by using definitions only. The insight provided by the Gelfand--Naimark duality is, in my opinion, well worth the trouble.

\subsubsection*{Acknowledgements} Special thanks to Bruno de Mendon\c a Braga and his assistant, and to Alessandro Vignati and ChatGPT Astra for advice on axiomatizability. 
 Many thanks to Alan Dow and Hannes Thiel for pointing out some typos, and to the anonymous referees whose remarks greatly improved the exposition. ChatGPT Astra was (fortunately) used for final proofreading. 
\section{Preliminaries}

All compacta are assumed to be nonempty. 
In this section we briefly mention Stone duality and then cover the basics of (commutative) \cstar-algebras and model theory. 

\subsection{Stone duality}

To a compact, Hausdorff, zero-dimensional space $X$ one associates $\Clop(X)$, the Boolean algebra of its clopen sets. If $f\colon X\to Y$ is a continuous function between compact Hausdorff spaces, then the function $f_*\colon \Clop(Y)\to \Clop(X)$ defined by 
\[
f_*(U)=f^{-1}(U)
\]
is a Boolean algebra homomorphism. 

\begin{theorem} \label{T.Stone.Duality} The functor $X\mapsto \Clop(X)$ is an equivalence of categories of compact, Hausdorff, zero-dimensional spaces and Boolean algebras. 
	This functor is contravariant. More precisely, we have the following for every continuous function $f\colon X\to Y$ between compact, Hausdorff, zero-dimensional spaces.  
	\begin{enumerate}
		\item $f$ is surjective if and only if $f_*$ is injective. 
		\item $f$ is injective if and only if $f_*$ is surjective. \qed 
	\end{enumerate}
\end{theorem}

This duality is very useful in the study of $\beta\bbN\setminus \bbN$ and \v Cech--Stone remainders of other zero-dimensional spaces (see \cite{Ru}, \cite[\S 0]{vM:Introduction}, \cite{dow2002applications}, \cite[\S 4]{Fa:AQ}).  

In the category of all compact Hausdorff spaces, the `Stone functor' is not injective---e.g., all continua are sent to the trivial two-element Boolean algebra. The lattice of closed sets serves as a surprisingly good replacement in the general case (e.g., \cite{DoHa:Universal}, \cite{hart2005there}). In this note I will try to convince the reader that another, more canonical, approach is well worth the trouble of learning the language. 

A related remarkable result is the  Gel'fand--Kolmogoroff theorem (\cite{gelfand1939on}) that connects compact Hausdorff spaces to their rings of continuous functions via the hull-kernel topology. By this theorem, the ring structure of a commutative \cstar-algebra already determines its full structure. To the best of my knowledge, this duality is not as compatible with the model-theoretic methods as the one to which this paper is devoted. 

\subsection{\cstar-algebras and Gelfand--Naimark duality}\label{S.GN}
A comprehensive and illuminating reference for \cstar-algebras and operator algebras in general is~\cite{Black:Operator}. See also \cite{Fa:STCstar} (this book will be, unapologetically, cited more often than necessary). 
Most readers can safely skip an initial segment of this subsection. 

\begin{definition}\label{D.C(X)}
To a compact Hausdorff space $X$ one can associate the Banach space\footnote{The functions in $C(X)$ are complex-valued rather than real-valued, because this gives us access to the  spectral theorem for bounded normal linear operators on a Hilbert space.} 
\[
C(X)=\{f\colon X\to \bbC: \text{$f$ is continuous}\}. 
\]
 When equipped with the supremum norm
\[
\|f\|=\sup_{x\in X} |f(x)|, 
\] 
pointwise multiplication, and conjugation, $C(X)$ becomes a commutative \emph{\cstar-algebra} with multiplicative unit $1$.
\end{definition}

\begin{remark} In \cite{gillmanrings}, the notation $C(X)$ is used for the ring of real-valued continuous functions on $X$, and the space $X$ is only assumed to be Tychonoff. 
	The same book uses $\cst(X)$ for the ring of real-valued, bounded, continuous functions on $X$. 
	Considering real-valued functions instead of complex-valued ones does not make much difference in the commutative setting. The subalgebra of real-valued functions in $C(X)$ (as in Definition~\ref{D.C(X)}) coincides with the self-adjoint part $C(X)_{\textrm{sa}}$ of this algebra (i.e., $f$ satisfying $f=f^*$) and is an algebra in its own right. Moreover, $C(X)$ is easily interpretable in $C(X)_{\textrm{sa}}$, as pairs $f+i g$ with the naturally defined arithmetic operations and norm. However, the self-adjoint part of a noncommutative \cstar-algebra is not even closed under multiplication. 
\end{remark}

The reader will notice that $C(X)$ is just the familiar $C_p(X)$ space (see \cite[\S 6]{vanMill2001infinite}), with two differences. The scalars are complex instead of real and is considered with respect to the uniform topology instead of the topology of pointwise convergence. 

 For readers rusty on functional analysis we provide some definitions.

\subsubsection{Hilbert space}
A (complex)\footnote{In case you haven't noticed yet, in this paper nothing is real.} Hilbert space $H$ is a Banach space equip\-ped with an \emph{inner product} $(\xi|\eta)$. This inner product is linear in the first coordinate and conjugate linear in the second (i.e., it is a \emph{sesquilinear form}, which is Latin for `one-and-a-half-linear'). The norm is defined as 
\[
\|\xi\|_2=(\xi|\xi)^{1/2}. 
\]
The following are some examples of Hilbert spaces. 
\begin{example} 
\begin{enumerate}
	\item If $Z$ is any set, then (writing $\bar a$ for $(a_z)_{z\in Z}$)
	\[
	\textstyle \ell_2(Z)=\{\bar a: a_z\in \bbC, \sum_z|a_z|^2<\infty\}
	\]
	equipped with the inner product\footnote{This sum is finite for all $\bar a$ and $\bar b$ by the Cauchy--Schwarz inequality.} $((\bar a)| (\bar b))=\sum_z a_z\overline {b_z}$ is a Hilbert space. 
	
\item 	The space ($\int$ stands for the Lebesgue integral)
	\[
	\textstyle L_2(\bbR)=\{f\colon \bbR\to \bbC: \text{ $f$ is Lebesgue-measurable and } \int |f|^2<\infty \}
	\]
	taken  modulo the equality almost everywhere and 
	equipped with the inner product $(f|g)=\int f\overline g$ (again, if $f$ and $g$ are in $L_2(\bbR)$, then the integral is finite by the Cauchy--Schwarz inequality). 
	
\item 	More generally, if $X$ is a locally compact Hausdorff space and $\mu$ is a strictly positive locally finite, $\sigma$-finite,  measure, then 
	\[
	\textstyle L_2(X,\mu)=\{f\colon X\to \bbC: \text{$f$ is measurable and }\int |f|^2\, d\mu<\infty\}
	\]
	taken  modulo the equality almost everywhere and 
	equipped with the inner product $(f|g)=\int f\overline g\, d\mu$ 
	is a Hilbert space. 
	
	One can take $\mu$ to be the counting measure on $\bbN$, in which case $L_2(\bbN,\mu)$ becomes $\ell_2(\bbN)$ as in the first example. 
\end{enumerate}

\end{example}

The space $\ell_2(\bbN)$ is isometrically isomorphic to $L_2(\bbR)$ and to every separable, infinite-dimensional Hilbert space. 

\subsubsection{Algebras}
Recall that a (complex) algebra is a complex vector space~$A$ equipped with a multiplication operation such that $(A,+,\cdot)$ is a ring. We will consider rings  equipped with a complete norm  that satisfies $\|ab\|\leq \|a\|\|b\|$. Therefore, the algebras that we are interested in will in particular be Banach algebras. All algebras in this text will be complex. The algebra $\cB(H)$ of bounded linear operators on a Hilbert space $H$ is an example of an algebra that has additional structure. 
It is equipped with the \emph{operator norm} 
\[
\|a\|=\sup\{\|a\xi\|_2: \xi\in H, \|\xi\|_2\leq 1\}
\]
and with the \emph{adjoint} operation, where the adjoint of $a$ is the uniquely defined operator $a^*$ such that for all $\xi$ and $\eta$ in $H$ we have 
\[
(a\xi|\eta)=(\xi| a^*\eta)
\] 

\begin{definition}A (concrete) \cstar-algebra is an algebra of operators on a complex Hilbert space $H$ that is closed in the operator norm and under taking adjoints.
A \emph{*-ho\-mo\-mor\-phi\-sm} between \cstar-algebras $A$ and $B$ is a map $\Phi\colon A\to B$ that is linear, multiplicative, and self-adjoint.\footnote{The category of \cstar-algebras is equipped with at least four natural morphisms, and `homomorphism' stands for one that preserves the algebra structure, but not necessarily the adjoint operation.}

An \emph{abstract \cstar-algebra} is a Banach algebra with involution $*$ that is isometrically\footnote{This condition is redundant; every *-homomorphism between \cstar-algebras is continuous and every isomorphism is an isometry (\cite[Lemma 1.2.10, but see Example 1.2.11]{Fa:STCstar}).} isomorphic to a \cstar-algebra on some Hilbert space. 
\end{definition}
By a result of Gelfand--Naimark--Segal, abstract \cstar-algebras have a simple axiomatization (\cite[Definition 1.2.2 and Theorem 1.10.1]{Fa:STCstar}).

\subsubsection{Compact Hausdorff spaces}

In Definition~\ref{D.C(X)} we saw that when $X$ is a compact Hausdorff space then the ring $C(X )$ is a (commutative) \cstar-algebra; we explore these algebras further.

\begin{example}
	Suppose $X$ is equipped with a strictly positive, finite,  $\sigma$-additive, Borel measure~$\mu$. (Recall that $\mu$ is \emph{strictly positive} if $\mu(U)>0$ for every nonempty open set $U$.) In this case $f\in C(X)$ can be identified with the multiplication operator $M_f$ on $L_2(X,\mu)$ defined by 
	\[
	M_f(g)=fg. 
	\]
	Proving that $M_f$ is a bounded linear operator on $L_2(X,\mu)$ and that $f\mapsto M_f$ defines an injective *-homomorphism 
	from $C(X)$ into\footnote{By a convention in the operator-algebraic literature, an isomorphism is not necessarily onto, hence this map is routinely called an isomorphism (into) in the literature.} an abelian \cstar-subalgebra of  $\cB(L_2(X,\mu))$
	is either obvious or a nice exercise, depending on the reader's disposition. 
\end{example}

If $f\colon X\to Y$ is a continuous function between compact Hausdorff spaces, then $f_*\colon C(Y)\to C(X)$ defined by 
\begin{equation}\label{eq.f*}
f_*(g)=g\circ f
\end{equation}
is a *-homomorphism. The following is a theorem of Gelfand and Naimark (\cite{gelfand1943imbedding}), proved five years after Wallman's \cite{wallman1938lattices}.

\begin{theorem} \label{T.GN.Duality} The functor $X\mapsto C(X)$ is an equivalence of categories of compact Hausdorff spaces and commutative, unital, \cstar-algebras whose morphisms are unital *-homomorphisms. 
This functor is contravariant. More precisely, we have the following for every continuous function $f\colon X\to Y$ between compact Hausdorff spaces.  
\begin{enumerate}
	\item $f$ is surjective if and only if $f_*$ is injective. 
	\item $f$ is injective if and only if $f_*$ is surjective. 
\end{enumerate}
\end{theorem}

\begin{proof} 
	This is presented in \cite[Theorem 1.3.1 and Theorem 1.3.2]{Fa:STCstar}, where additional details can be found. Nevertheless, the main ideas of the proof are well worth outlining (mathematics is, after all, all about proofs).  
	Given a commutative, unital \cstar-algebra $A$, one needs to recover the compact Hausdorff space $X$ such that $A\cong C(X)$. This is the space $\hat A$ of all nonzero *-homomorphisms $\varphi\colon A\to \bbC$; such $\varphi$ is called a \emph{character}.\footnote{This definition makes sense even if $A$ is not unital; see Definition~\ref{D.Gelfand} below.}  Every character is automatically self-adjoint, continuous, and of norm 1, hence the space $\hat A$ of all characters of $A$ is a subspace of the unit ball of the Banach dual of $A$. By the Banach--Alaoglu theorem, the dual unit ball is compact with respect to the weak*-topology. Its subspace $\hat A$ is  closed in the weak*-topology (e.g., if $|\varphi(xy)-\varphi(x)\varphi(y)|>0$ then $\varphi$ has a weak*-open neighbourhood consisting of functionals that satisfy the same inequality). Therefore  $\hat A$ is compact with respect to the weak*-topology. 
	
Every $a\in A$ defines a weak*-continuous function $f_a$ on $\hat A$, the evaluation functional. The function $a\mapsto f_a$ is clearly a *-homomorphism of $A$ into $C(\hat A)$. 
	A (nontrivial) Hahn--Banach argument\footnote{As pointed out in \cite[p. 18]{BrOz:C*}, `Hahn--Banach theorem' really  refers to a collection of related results and corollaries, and some authors assume familiarity with all possible formulations and don't bother explaining which version is invoked.} implies that for every nonzero $a\in A$ there is a character $\varphi$ such that $a$ is not in the kernel of $\varphi$. 
 Therefore $a\mapsto f_a$ is an injective *-homomorphism of $A$ into $C(\hat A)$ and it is by \cite[Corollary~1.2.11]{Fa:STCstar}  isometric. Therefore the image of $A$ is norm-closed. 
 Since it is an algebra closed under the adjoint operation that contains all constant functions and separates points, by the complex Stone--Weierstrass theorem this *-homomorphism is surjective.

	To prove that $X\mapsto C(X)$ is an equivalence of categories, we need to verify that every unital *-homomorphism $\Psi\colon C(Y)\to C(X)$ is of the form~$f_*$ (see~\eqref{eq.f*}) for some continuous $f\colon X\to Y$. Using the notation from the previous paragraph, $X=\widehat{C(X)}$ and $Y=\widehat {C(Y)}$, and 
\[
f(\varphi)=\varphi\circ \Psi
\] 
defines a function from the space of characters of $C(X)$ to the space of characters of $C(Y)$. It is straightforward to verify, using the continuity of~$\Psi$, that $f$ is weak*-continuous.   

	Now assume $f\colon X\to Y$. 
The fact that if $f$ is injective then $f_*$ is surjective is the Tietze extension theorem. 
If $f$ is surjective, then $f_*$ is injective because continuous functions on $Y$ separate points.

Finally, if $f_*$  is onto, then functions separating points of $X$ show that~$f$ is injective. Conversely,  if $f_*$ is injective but $f$ is not surjective , then  a nonzero continuous function vanishing on the compact set $f[X]$ gives a contradiction.
\end{proof}

The construction of the space $X$ from an abstract unital abelian \cstar-algebra $A$ is the main content of the proof of  Theorem~\ref{T.GN.Duality}.

\begin{definition} \label{D.Gelfand} If $A$ is an abelian (not necessarily unital) \cstar-algebra, then the space $\hat A$ of all characters of $A$ is called the \emph{Gelfand spectrum} of $A$. 
\end{definition}

By Theorem~\ref{T.GN.Duality}, if $A$ is unital and abelian then $A\cong C(\hat A)$. If $A$ is abelian but not necessarily unital, then the proof of Theorem~\ref{T.GN.Duality} shows that  $\hat A$ is locally compact and $A\cong C_0(\hat A)$.

It is worth remarking that the Stone duality is a consequence of the Gelfand--Naimark duality (see Lemma~\ref{L.Projections} below).

\subsection{Continuous functional calculus}

This subsection is devoted to a consequence of the Gelfand--Naimark duality. In this paper it will only be used  in \S\ref{S.Axiomatizable} to demonstrate that certain properties of compacta are co-axiomatizable  and to show that the Stone duality is a consequence of the Gelfand--Naimark duality. 

\begin{definition}
The \emph{spectrum} of an element $a$ of a unital \cstar-algebra $A$ is defined as (by $1_A$ we denote the unit of $A$)
\[
\Sp(a)=\{\lambda\in \bbC: a-\lambda 1_A\text{ is not invertible}\}. 
\]
If $A$ is non-unital and $a\in A$, then $\Sp(a)$ is defined as its spectrum in the unitization of $A$. 
\end{definition}

By a theorem of Gelfand, this is a nonempty compact subset of $\bbC$ for every operator $a$ (e.g., \cite[Theorem~1.2.5]{Murphy:C*}). 
Although the definition of $\Sp(a)$ apparently depends on the ambient \cstar-algebra, if $A$ is a unital \cstar-subalgebra of $B$ (meaning that they share the same unit) and $a\in A$, then $\Sp(a)$ as computed in $A$ is equal to $\Sp(a)$ as computed in $B$ (see \cite[Exercise~1.11.6]{Fa:STCstar} and the supplied hint). 

An element $a$ of a \cstar-algebra is normal if it commutes with its adjoint. Clearly in an abelian \cstar-algebra all elements are normal, and the converse is also true: if all elements of $A$ are normal, then $A$ is commutative.  The reason for this is that every noncommutative \cstar-algebra contains a nilpotent element (see \cite[Proposition~3.6.6]{Fa:STCstar} for this and other equivalent reformulations). 

\begin{theorem}
	[Continuous functional calculus]\label{T.CFC}  If $a$ is a normal element of a unital \cstar-algebra, then $\cst(a,1_A)$, the \cstar-subalgebra generated by $a$ and~$1_A$, is isomorphic to $C(\Sp(a))$.
The natural isomorphism sends $a$ to the identity function on $\Sp(a)$. 
\end{theorem}

\begin{proof}
	Every character $\varphi$ of $\cst(a,1)$ is uniquely determined by $\varphi(a)$, and since $\varphi(1_A)=1$ necessarily $\varphi(a)\in \Sp(a)$. Conversely, every $\lambda\in \Sp(a)$ uniquely determines a character that sends $a$ to $\lambda$. The map $\varphi\mapsto \varphi(a)$ is a weak*-continuous bijection between compact Hausdorff spaces, hence a homeomorphism.  The statement now follows from Theorem~\ref{T.GN.Duality}. 
\end{proof}

This is a good moment to introduce some types of operators in \cstar-algebras, corresponding to properties of bounded linear operators on a Hilbert space (for more see \cite[Definition~1.4.1]{Fa:STCstar}). 

\begin{definition}\label{D.Operators}
	If $A$ is a \cstar-algebra, then some $a\in A$ is called
	\begin{enumerate}
		\item self-adjoint if $a=a^*$;
		\item positive if $a=b^*b$ for some $b\in A$;\footnote{This formulation is the tip of an iceberg; see \cite[\S 1.6]{Fa:STCstar}.} 
		\item projection if $a=a^*$ and $a=a^2$. 
	\end{enumerate}
\end{definition}

The following lemma holds for normal elements  of any \cstar-algebra, but the proof of part \eqref{2.sa.positive.projection}  is more involved than we are prepared to undertake at the moment (and we fortunately don't need this case). 

\begin{lemma}\label{L.sa.positive.projection}
	Suppose that $A$ is a  commutative \cstar-algebra. Then the following holds for all $a\in A$. 
	\begin{enumerate}
		\item $a$ is self-adjoint if and only if $\Sp(a)\subseteq \bbR$. 
		\item\label{2.sa.positive.projection} $a$ is positive if and only if $\Sp(a)\subseteq [0,\infty)$. 
		\item $a$ is a projection if and only if $\Sp(a)\subseteq \{0,1\}$. 
	\end{enumerate}
\end{lemma}

\begin{proof}
	Apply Theorem~\ref{T.CFC}. By the continuous functional calculus, we have that $a=a^*$ if and only if the identity function on $\Sp(a)$ has this property, which is equivalent to $\Sp(a)\subseteq \bbR$. 
	For the second part, note that $(b^*b)^*=b^*b$, hence every positive element $a$ is self-adjoint. Since $A$ is commutative, $b$ is normal  and by the continuous functional calculus applied to $b$ every element of the spectrum of $a$ has the form $|\lambda|^2$ for some $\lambda\in \bbC$. This is equivalent to $\Sp(a)\subseteq [0,\infty)$. Conversely, if $\Sp(a)\subseteq [0,\infty)$  then by the continuous functional calculus we can take $b=a^{1/2}$ so that $a=b^*b$. 
	
	Finally, $a^*=a=a^2$ is equivalent $\overline \lambda=\lambda=\lambda^2$ for all $\lambda\in \Sp(a)$,  hence a normal element satisfies  these  equations if and only if $\Sp(a)\subseteq \{0,1\}$. 
\end{proof}

The part of Lemma~\ref{L.sa.positive.projection} which is most useful for our purposes is the last one. It shows that projections in a commutative \cstar-algebra $C(X)$ correspond to characteristic functions of clopen subsets of $X$, as these are the only subsets of $X$ whose characteristic functions belong to $C(X)$.

As promised earlier, we can now show how the Gelfand--Naimark duality implies the Stone duality. 

\begin{lemma}
	\label{L.Projections} If $X$ is a compact Hausdorff space, then the Boolean algebra $\Clop(X)$ of its clopen subsets is equal to the Boolean algebra of projections in $C(X)$.  If $X$ is zero-dimensional, then $C(X)$ is the completion in the sup-norm of the algebra of all linear combinations of its projections. \qed 
\end{lemma}

Theorem~\ref{T.CFC} has a natural generalization to finite sets of commuting operators (\cite[Exercise~1.11.22]{Fa:STCstar}). 

\subsection{Separability}
The following, together with the L\"owenheim--Skolem theorem,  will come handy later on. 

\begin{lemma}\label{L.MetrizableSeparable} Suppose that $X$ is compact and Hausdorff. 
	A commutative \cstar-algebra $C(X)$ is separable if and only if~$X$ is metrizable. 
	More generally, if $X$ is infinite then the density character of $C(X)$ is equal to the weight of~$X$. 
\end{lemma}

\begin{proof} We will prove only the first part since the proof of the second part is analogous. 
	If $A$ is a unital, separable, Banach space then the unit ball of its dual is compact and metrizable (e.g., \cite[E.2.5.3]{Pede:Analysis}) in the weak*-topology, and $X$ is a closed subspace of the dual unit ball of $C(X)$.

	For the converse, if $X$ is compact and metrizable then a countable family $[0,1]$-valued functions separates points of $X$. The $\bbQ+i\bbQ$ *-algebra generated by this set and the constant functions is, by the complex Stone--Weierstrass theorem, dense in $C(X)$. 
\end{proof}

Later on we will see that Theorem~\ref{T.GN.Duality}  generalizes Stone duality (Theorem~\ref{T.Stone.Duality}), as a consequence of Lemma~\ref{L.ProjectionsAreDefinable}. 

\subsection{Locally compact Hausdorff spaces}
If $X$ is a locally compact, non-compact, Hausdorff space, then 
\[
C_0(X)=\{f\colon X\to \bbC: f\text{ is continuous and vanishes at infinity}\}
\]
is a commutative, non-unital, \cstar-algebra, and every commutative, non-unital, \cstar-algebra is of this form. 
One does not, however, have an equivalence of categories, because not every *-homomorphism between $C_0(Y)$ and $C_0(X)$ corresponds to a continuous function from $X$ to $Y$ (\cite[II.2.2.7]{Black:Operator}). 

There are two standard constructions that embed a non-unital \cstar-algebra into a unital one, the minimal and the maximal one. 

\begin{definition} Suppose that $A$ is a non-unital \cstar-algebra. 
	The \emph{unitization} of $A$, denoted 
	$A^\dagger$, is the algebra defined on the Banach space $A\oplus \bbC$ 
	of formal sums $a+\lambda\cdot  1$ 
	for $a\in A$ and $\lambda\in \bbC$. In case when $A=C_0(X)$, the norm is given by $\|(a,\lambda)\|=\sup_{x\in X} |a(x)+\lambda|$.     The adjoint,  addition, and scalar multiplication are defined pointwise,  and 
	\[
	(a+\lambda\cdot 1) (b+\mu\cdot 1):=ab+\lambda b + \mu a+ \lambda\mu\cdot 1. 
	\]
We identify $A$ with the subalgebra $a+0\cdot 1_A$, for $a\in A$. 
\end{definition}

Suppose that $A$ is a non-unital, commutative \cstar-algebra. Then $A^\dagger$ is isomorphic to $C(X)$ for the compact Hausdorff space $X=\widehat {A^\dagger}$. Straightforward computation now shows that $A$ is isomorphic to $C_0(X\setminus \{\varphi\})$, where $\phi(a+\lambda 1_A)=\lambda$. 

The Gelfand spectrum of $A^\dagger$ (Definition~\ref{D.Gelfand}) is naturally homeomorphic to the one-point compactification of the spectrum of $A$, with the character that annihilates $A$ corresponding to the point at $\infty$. 

\begin{definition}
Now suppose that $X$ is a completely regular topological space and let 
\[
C_b(X)=\{f\colon X\to \bbC:f\text{ is continuous and bounded} \}. 
\]
 \end{definition}

Readers familiar with \cite{gillmanrings} should take note of two notational differences. While in \cite{gillmanrings} all functions are real-valued, in our case they are complex-valued (this convention allows us to use the spectral theorem for normal operators). Also, in \cite{gillmanrings}  $C(X)$ stands for all (possibly unbounded) continuous functions and $\cst(X)$ denotes the space of bounded continuous functions, denoted $C_b(X)$ here. (In the operator algebra literature, writing $C(X)$ implicitly assumes that $X$ is compact Hausdorff and writing $C_0(X)$ implicitly assumes that it is locally compact Hausdorff.)

By the assumed complete regularity of $X$, the functions in $C_b(X)$ separate points of $X$. It is straightforward to check that $C_b(X)$ is a unital, commutative, \cstar-algebra. Its spectrum $Z$ is a compactification of $X$ with the property that every continuous, bounded,  complex-valued (or real-valued) function has a unique continuous extension to $Z$. By one of the well-known characterizations of  the \v Cech--Stone compactification of $X$, $\beta X$, we have the following.  

\begin{lemma}
	If $X$ is a completely regular Hausdorff space, then the Gelfand spectrum of $C_b(X)$ is homeomorphic to $\beta X$. \qed 
\end{lemma}

The algebra $C_b(X)$ is a special case of an important construction, usually  considered to be the non-commutative analog of the \v{C}ech--Stone compactification.  

\begin{lemma}\label{L.Multiplier}
	The \emph{multiplier algebra} of a commutative \cstar-algebra $C_0(X)$ is defined 
	as 
	\[
	\cM(A)=\{g\colon X\to \bbC: gf\in C_0(X)\text{ for all }f\in C_0(X)\}. 
	\]
The multiplier algebra of $C_0(X)$ is the algebra 
$C_b(X)$ of all bounded, complex-valued, continuous functions on $X$. 
\end{lemma}

\begin{proof}
Clearly $C_b(X)\subseteq \cM(C_0(X))$, and we need only prove the converse inclusion. 

Fix $g\colon X\to \bbC$. 
Since $X$ is locally compact, if $g$ is discontinuous then there is a compactly supported and continuous $f\colon X\to \bbC$ such that $fg$ is discontinuous. 
We may therefore assume $g$ is continuous.  Now suppose $g$ is unbounded and fix a sequence $x_n\in X$ such that  $|g(x_n)|>n$, for $n\in \bbN$,  and the sequence is strictly increasing. Then this sequence escapes every compact subset of $X$ and we can define
$f_0(x_n)=\sqrt n/g(x_n)$, and $f_0(\infty)=0$  in the one-point compactification of $X$. Then $f_0$ is a continuous function on a closed subset of the one-point compactification of $X$. By the Tietze extension theorem, fix a continuous extension $f\in C_0(X)$ of $f_0$. Then $fg$ is unbounded. 
This completes the proof that $\cM(C_0(X))\subseteq C_b(X)$. 
\end{proof}

\begin{funfact} A subspace $Y$ of  $X$ is called \cstar-embedded  (\cite[1.16]{gillmanrings}) if 
	\[
	C_b(Y)=\{f\rs Y: f\in C_b(X)\}.
	\] 
	Although in \cite{gillmanrings} the scalars are real, since every complex-valued function corresponds to two real-valued functions, this makes no difference. 
\end{funfact}

\begin{remark}
	Curious readers may wonder what the definition of a multiplier algebra of a non-unital and non-commutative \cstar-algebra $A$ looks like. Although we will not need it in this article, here it is.  
	One needs to fix a representation $\pi$ of $A$ on a Hilbert space $H$ (that is, $\pi\colon A\to \cB(H)$ is a *-homomorphism). The representation $\pi$ is \emph{faithful} if its kernel is trivial (i.e., equal to $\{0\}$; equivalently, if $\pi$ is isometric) and \emph{non-degenerate} if the annihilator 
	\[
	\{b\in \cB(H): ba=ab=0\text{ for all $a\in A$}\}
	\] 
	of $\pi[A]$ is trivial (again, $\{0\}$). 
	Using a faithful and non-degenerate $\pi$, the \emph{multiplier algebra} $\cM(A)$ of $A$ is   
	defined as the \emph{idealizer} of $\pi[A]$, 
	\[
	\cM(A)=\{b\in \cB(H): ba,ab\in A\text{ for all }a\in A\}. 
	\]
	This is only one of the ways to define $\cM(A)$, and it can be proved that~$\cM(A)$ and the embedding of $A$ into it do not depend on the choice of the representation. See \cite[\S 13]{Fa:STCstar} for details and \cite[II.7.3]{Black:Operator} for alternative definitions of the multiplier algebra. 
\end{remark}

\begin{lemma}\label{L.ProjectionsAreDefinable}
	If $X$ is a locally compact Hausdorff space, then $\cM(C_0(X))$ is isomorphic to $C(\beta X)$, the quotient $\cM(C_0(X))/C_0(X)$ is isomorphic to $C(\beta X\setminus X)$. If $\iota$ is the identity map from $\beta X\setminus X$ into $\beta X$, then its dual is the quotient map $q\colon \cM(C_0(X))\to \cM(C_0(X))/C_0(X)$. \qed 
\end{lemma}

\section{Model theory} Over the years, there were many connections between model theory and topology, from topological space-valued logic (\cite{chang1966continuous}) and  first-order logic of topological spaces (\cite{henson1977first}), to 
Bankston's study of ultraco-products of compact Hausdorff spaces  (\cite{bankston2000some}, \cite{bankston1997co}, \cite{Bankston1987reduced}, \cite{bankstonsurvey}).  
Elementary submodels of models of large enough fragments of ZFC have been used in topology for decades (e.g., \cite{dow1988introduction}, \cite{dow1992set},  \cite{kunen2003compact}, \cite{bandlow1991characterization}, \cite{bandlow1991construction}, \cite{balogh1996small}, 
\cite{geschke2002applications},  \cite{vandersteeg2003models}, \cite{bartovsovaa2011lelek},  \cite{eisworth2006elementary}, \cite{soukup2018infinite}).

We will use continuous model theory (\cite{BYBHU}, \cite{hart2023an}), as adapted to \cstar-algebras in~(\cite{Muenster}). 
The reader should be warned that the continuous logic as introduced in \cite{BYBHU} is the logic of bounded metric structures. Its standard adaptation to unbounded metric structures such as \cstar-algebras is explained, for example, in \cite[Example 3.10]{hart2023an} and in \cite[\S 2]{Muenster}. 
A quick introduction follows; for more details see any of the above references.

\subsection{Syntax and semantics}
\label{S.Syntax.LMS}

Our language consists of function symbols 
$+$, $\cdot$, ${}^*$,   constant symbols $0$ and $1$,  a separate unary function for multiplication by each complex number, and the norm symbol $\|\cdot\|$. 
For every $n$ there is an infinite supply of variables that range over the $n$-ball, and every variable is assigned to some $n$-ball. 
\begin{enumerate}
	\item \label{I.Term} \emph{Terms} are expressions of the form $\|P(\bar x)\|$, where $\bar x$ is a tuple of variables and $P$ is a complex *-polynomial in commuting variables.
\end{enumerate}  
\emph{Quantifier-free formulas} are expressions of the form 
\begin{enumerate}[resume]
	\item\label{I.QF} $f(t_0,\dots, t_{n-1})$
\end{enumerate}
for $n\geq 1$, terms $t_j$, for $j<n$, and a continuous function $f\colon \bbR^n\to \bbR$.

The set of formulas is the smallest set $\bfF$ that satisfies the following. 
\begin{enumerate}[resume]
	\item All quantifier-free formulas belong to $\bfF$. 
	\item \label{I.Quantifiers} If $\varphi$ belongs to $\bfF$, $x$ is a variable, and $m\geq 1$, then $\sup_{\|x\|\leq m} \varphi$ and $\inf_{\|x\|\leq m}\varphi$ belong to $\bbF$. 
	\item If $n\geq 1$, $\varphi_j$ for $j<n$ belongs to $\bfF$, and $f\colon \bbR^n\to \bbR$ is continuous, then $f(\varphi_0,\dots, \varphi_{n-1})$ belongs to $\bfF$. 
\end{enumerate}

Functions $f$ as in \eqref{I.QF} correspond to logical connectives,
$\vee$, $\wedge$, and $\leftrightarrow$. 
(There are no exact analogs of negation or implication in the  logic of metric 
structures, although the truncated subtraction $x\dminus y=\max(0,x-y)$ often serves as a poor man's replacement for the latter.)
The `quantifiers' $\sup$ and $\inf$ in \eqref{I.Quantifiers} correspond to the quantifiers $\forall$ 
and $\exists$.  

	Free (and bound) variables in a formula are defined as usual, by induction on the complexity of the formula. 
A formula with no free variables is called a \emph{sentence}\index{sentence!metric}. 	The set of all $\calL$-formulas whose free variables 
are included in $\bar x = (x_0,\ldots,x_{n-1})$
is denoted~$\bfF^{\bar x}$.

\subsection{The space of formulas}
In discrete logic, sentences of a language $\calL$ form a Boolean algebra called the Lindenbaum algebra. Ultrafilters of this algebra correspond to complete $\calL$-theories. Let's see what are the analogous objects in continuous logic. 

\begin{definition}\label{D.DefinablePredicates}
If $\calL$ is a metric language, then the space of $\calL$-sentences  is by definition closed under addition, multiplication, and multiplication by real scalars\footnote{Warning: This is not the case in the original bounded version in which the formulas are $[0,1]$-valued. This restriction is natural when one deals with bounded metric structures, but for Banach space-based objects such as \cstar-algebras it is too restrictive.}  and therefore forms an  algebra over $\bbR$. This algebra is equipped with a seminorm
\[
\|\varphi\|=\sup \{|\varphi^M|: M\text{ is an $\calL$-structure}\}. 
\]
By quotienting and completing, we obtain a real Banach algebra. The elements of this Banach algebra are (equivalence classes of) limits of $\|\cdot\|$-Cauchy sequences of $\calL$-formulas, and some authors conveniently consider these as $\calL$-formulas.  They are also called definable predicates. 
Characters (i.e., unital homomorphisms into $\bbR$) of this algebra correspond to complete $\calL$-theories. 

Suppose that $\bar x$ is a fixed tuple of variables (possibly countable) and recall that the syntax dictates that every variable ranges over a fixed bounded ball.  Then analogous construction shows gives a Banach algebra of $\calL$-definable predicates whose  free variables are included among those in $\bar x$. 
\end{definition}

\begin{remark}[Separability of the language]\label{R.Separability}
	Every continuous language has uncountably many formulas, since we allow all continuous real functions as connectives. However, it suffices to consider a countable set $\cF$ of connectives and only $\cF$-restricted formulas (see \cite[\S 6]{BYBHU}). 
	Suppose that  (in addition to the scalar multiplications) there are only countably many non-logical symbols (as is the case for the language of \cstar-algebras). Then multiplication by an arbitrary complex number can be uniformly approximated in the norm $\|\cdot\|$ by multiplication by complex rationals, and   the Banach algebra of definable predicates associated to this language is separable. In this case, we say that the language is separable. 
\end{remark}

\begin{definition}\label{D.Theory} The theory of $C(X)$, $\Th(C(X))$, is the functional of the algebra described in the previous paragraph that sends every sentence $\varphi$ to~$\varphi^{C(X)}$.\footnote{An alternative definition of a theory is as the kernel of this functional.  The two definitions are interchangeable (see e.g., \cite[\S D.2]{Fa:STCstar}), and the latter one makes it easier to talk about theories that are not complete, as in \S\ref{S.Axiomatizable} below.}
	
	If $X$ is a compact Hausdorff space, then the \emph{co-theory of~$X$, $\coTh(X)$}, is defined to be $\Th(C(X))$, the theory of $C(X)$. 
\end{definition}

\subsection{Co-axiomatizable properties of compacta I. Explicit formulas}\label{S.Axiomatizable}
A class $\cC$  of \cstar-algebras is elementary (or axiomatizable)  if there is a set of sentences $T_\cC$ such that $A\in \cC$ if and only if $\varphi^A=0$ for all $\varphi\in T_\cC$.  (Equivalently, $T_\cC$ is included in the kernel of the functional $\Th(A)$ if and only if $A\in \cC$.)
A property $\bbS$ of compacta is called co-axiomatizable if the property `$\hat A$ has property $\bbS$' of \cstar-algebras is axiomatizable. 

Let us first consider axiomatizability of connectedness.  A space $X$ is connected if and only if its only clopen subsets are $X$ and $\emptyset$. In case when $X$ s compact Hausdorff, by Lemma~\ref{L.sa.positive.projection} and the remark following it, this is equivalent to asserting that the only projections in $C(X)$ are $0$ and $1$. The syntax of continuous logic does not allow us to quantify over arbitrary subsets of the structure at hand such as its projections.  
While arbitrarily extending the syntax typically results in an ill-behaved logic, the set of projections is very special. 

We extend the syntax by allowing quantification $\sup$ and $\inf$ over elements that are self-adjoint, positive, or projections which belong to a fixed bounded ball.  Let $\calL_+$ denote this expanded language. 

We say that $\calL_+$-formulas $\varphi(\bar x)$ and $\psi(\bar x)$ are equivalent if for every $M$ and $\bar a$ in $M$ of the appropriate sort, $\varphi^M(\bar a)=\psi^M(\bar a)$.
See Definition~\ref{D.Operators} for self-adjoint, positive, and projections. We abbreviate the phrase `$x$ is positive' by `$x\geq 0$'. 

\begin{lemma}\label{L.Definable} In the following $Q$ stands for a quantifier, $\sup$ or $\inf$, and $\varphi(x)$ is an arbitrary $\calL_+$-formula (possibly with other free variables). 
\begin{enumerate}
	\item The formulas $Q_{\|x\|\leq n, x=x^*}\varphi(x)$ and $Q_{\|x\|\leq n}\varphi((x+x^*)/2)$ are equivalent. 
	\item The formulas $Q_{x\geq 0, \|x\|\leq n^2}\varphi(x)$ and $Q_{\|x\|\leq n}\varphi(x^*x)$ are equivalent. 
\end{enumerate}
\end{lemma}

\begin{proof}
	Fix a \cstar-algebra $A$ (not necessarily commutative). Then the set of its self-adjoint elements is equal $\{(a+a^*)/2: a\in A\}$, and the set of its positive elements is equal to $\{a^*a: a\in A\}$.
\end{proof}

Recall that the elements of the completion of the space of all formulas are called definable predicates (Definition~\ref{D.DefinablePredicates}).  

\begin{lemma} \label{L.Definable.Projection}  If $Q$ stands for $\sup$ or $\inf$, then  $Q_{\xproj}\varphi(x)$ is a definable predicate  for every $\calL_+$ formula $\varphi(x)$. 
\end{lemma}

\begin{proof} As this proof requires a fair amount of machinery, we only provide a reference. By \cite[Example~3.2.7 (5)]{Muenster}, the set of all projections is definable in a \cstar-algebra and by \cite[Lemma 3.2.5]{Muenster} this implies that if $\varphi(x)$ is a definable predicate then so is 
$Q_{\xproj}\varphi(x)$. The proof now proceeds by induction on complexity of $\varphi$. 
\end{proof}

Lemma~\ref{L.Definable.Projection} is all we need from the exciting theory of definable sets in continuous logic and its relation to axiomatizability. For more details and the proper definition of a definable set, see \cite[\S 3.2]{Muenster}. 

In continuous logic, the negation of an axiomatizable property is not necessarily axiomatizable. 

Here is a convenient bit of notation, coming from \cite{BYBHU}. For reals $x,y$ let
\[
x\dminus y = \max(0,x-y). 
\]

\begin{proposition}\label{P.Connected}
	Both properties `$X$ is connected' and `$X$ is not connected' of compact Hausdroff spaces are co-axiomatizable. 
\end{proposition}

\begin{proof}
	By using Lemma~\ref{L.Definable.Projection}, it suffices to find sentences in $\calL_+$ that axiomatize these properties. Consider the following. 
	\begin{align*}
	\psi: & \sup_{\xproj} \min(\|x\|, \|1-x\|),\\
	\varphi: & \inf_{\xproj} 1\dminus \min(\|x\|, \|1-x\|).
	\end{align*}
	Then $\psi^{C(X)}=0$ if and only if $1$ and $0$ are the only projections in $C(X)$, if and only if $X$ is connected. 
	On the other hand, $\varphi^{C(X)}=0$ if and only if there is a projection $p\in C(X)$ such that $\|p\|=\|1-p\|=1$. This is equivalent to $X$ being disconnected. 
\end{proof}

The axiomatization given in Proposition~\ref{P.Connected} does not apply to general locally compact Hausdorff spaces $X$, because the projections in $C_0(X)$ are compact clopen subsets. For example $\bbR\sqcup \bbR$ is not connected although $C_0(\bbR\sqcup \bbR)$ has no nonzero projections. 

\begin{proposition} The property `$X$ is zero-dimensional' of compacta is axiomatizable. 
\end{proposition}

\begin{proof}
	Consider the $\calL_+$-sentence
	\[
	\varphi_{\text{0-dim}}: \qquad \sup_{\|x\|\leq 1}\sup_{\|y\|\leq 1} \inf_{\zproj} \max(\|zx\|,\|(1-z)y\|)^2\dminus \|xy\|. 
	\]
It is a nice exercise to verify that $\varphi^{C(X)}=0$ is equivalent to zero-di\-men\-sio\-na\-lity. 
\end{proof}

The formula $\varphi_{\text{0-dim}}$ (with $x,y$ required to range over self-adjoint elements) also axiomatizes the important notion of having real rank zero, which is the generalization of zero-dimensionality to the non-commutative context (see \cite[\S 3.6.2]{Muenster}). 
Axiomatizability of the two generalizations of dimension to the non-commutative context, real rank and stable rank (see \cite[\S V.3]{Black:Operator}), has been studied in \cite[\S 3.8, 3.9]{Muenster} and \cite[Lemma~6]{farah2016axiomatizability}, where an explicit formula $\varphi_n$ such that $\sr(A)\leq n$ if and only $\varphi_n^A=0$ and $\sr(A)>n$ if and only if $\varphi_n^A=1$. Unfortunately, $\sr(C(X))= \lfloor \dim(X)/2\rfloor +1$, hence $\dim(X)$ cannot be recovered exactly from $\sr(C(X))$. It is however very likely that the methods of \cite{farah2016axiomatizability} can be modified to give an explicit axiomatization of the value of the dimension of~$X$. This is indeed the case in the Wallman compactification approach, where explicit formulas for the dimension of $X$ exist.

\subsection{Expanding the language}

We sometimes expand the language as follows. 
\begin{definition}[Expanded language]\label{D.Expanded}
Given a \cstar-algebra~$A$, we define terms and formulas in the expanded language $\calL_{A}$ as follows. 
For each element of $A$ we introduce a constant (routinely and harmlessly identified with the element itself), and replace \eqref{I.Term} with the following. 
\begin{enumerate}[resume]
	\item \label{I.Term.expanded} \emph{Terms} in $\calL_A$ are expressions of the form $\|P(\bar x,\bar a)\|$ where $\bar x$ is a tuple of variables, $\bar a$ is a tuple of elements of $A$, and $P(\cdot, \cdot)$ is a complex *-polynomial in commuting variables.
\end{enumerate}  
\end{definition}

\subsection{Reduced co-products, ultraco-products, and co-theories}\label{S.ultracoproducts}
For a history of ultraproducts, merge the introductions  to \cite{bankstonsurvey} and \cite{She:Notes}. 
If $\cF$ is a filter on an index set $\bbI$ and $A_i$, for $i\in \bbI$, are \cstar-algebras, then we define the reduced product $\prod_\cF A_i$ as follows (see \cite[\S 16.2]{Fa:STCstar} for details).

Given a family of \cstar-algebras $A_i$, for $i\in \bbI$, the product is defined as the $\ell_\infty$-product rather than the Cartesian product (we'd rather have the product be a \cstar-algebra!): 
\[
\textstyle \prod_i A_i=\{(a_i): a_i\in A_i\text{ for all $i$ and $\sup_i \|a_i\|<\infty$}\}. 
\]
With $\limsup_{i\to \cF} r_i=\inf_{X\in \cF} \sup_{i\in X} r_i$, let 
\[
c_\cF=\{(a_i): \limsup_{i\to \cF} \|a_i\|=0\} .
\]  
This is a two-sided, norm-closed, self-adjoint ideal. Let
\[
\textstyle \prod_\cF  A_i=\prod_i A_i/c_\cF. 
\]
It should be noted that the fact that the quotient of a \cstar-algebra over a two-sided, norm-closed, self-adjoint ideal is a \cstar-algebra is nontrivial (\cite[Lemma~2.5.2]{Fa:STCstar}).

\begin{definition} If $\cF$ is a filter on an index set $\bbI$ and $X_i$, for $i\in \bbI$, is a family of compact Hausdorff spaces, then the \emph{reduced co-product} $\sum_\cF X_n$ is defined to be the Gelfand spectrum of the reduced product $\prod_\cF C(X_n)$. 

	If $\cI$ is an ideal on $\bbI$, then $\sum_{\cI} X_i$ denotes the reduced co-product associated with the filter dual to $\cI$, 
	\[
	\cI_*=\{A\subseteq \bbI: \bbI\setminus A\in \cI\}. 
	\]
	\end{definition}

In the following $\Fin$ denotes the  Fr\'echet filter on a fixed index set $\bbI$. 

\begin{lemma}\label{L.betaX}
	For a family of compact Hausdorff spaces $X_i$, for $i\in \bbI$, the reduced co-product $\sum_{\Fin} X_i$ is homeomorphic (arguably, equal) to the \v Cech--Stone remainder of the discrete sum $\bigsqcup_i X_i$. 
If $\cF$ is a filter on $\bbI$ that includes all cofinite sets, then $\sum_\cF X_i$ is homeomorphic (arguably, equal) to a closed subspace of the \v Cech--Stone remainder of the discrete sum $\bigsqcup_i X_i$. 
\end{lemma}

\begin{proof} Let $X=\bigsqcup_i X_i$. Then every bounded continuous function on $X$ has a unique extension to a continuous function on $\beta X$, hence $C(\beta X)$ is naturally identified with $\prod_i C(X_i)$.\footnote{Remember, this was defined as the $\ell_\infty$ product.} 
	By Lemma~\ref{L.Multiplier}, $\prod_i C(X_i)$  is naturally identified with the multiplier algebra of $\bigoplus_i C(X_i)$.  
	Apply Theorem~\ref{T.GN.Duality}.  
\end{proof}

If $\cU$ is an ultrafilter, then we talk about ultraproducts and ultraco-products. 
The notation $\sum_\cF X_i$ is taken from \cite{Bankston1987reduced}, where this notion was introduced. 
 This definition did not use continuous model theory,  which did not exist at the time. The fact that the two definitions of ultraco-products coincide serves as a testimony to the unreasonable effectiveness of continuous logic. 

The following direct construction of $\Sigma_\cF X_i$ can be extracted, more or less, from the proof of \cite[Theorem~3.1]{dow1993cech}. 

\begin{lemma}
For a family $X_i$, for $i\in \bbI$, of compact Hausdorff spaces and a filter $\cF$ on $\bbI$ (the closure of $\bigcup_{i\in A} X_i$ is taken inside the \v Cech--Stone compactification of the discrete sum $\bigsqcup_i X_i$),	
\[
\Sigma_\cF X_i\text{ is homeomorphic to }\textstyle  \bigcap_{A\in \cF} \overline{\bigcup_{i\in A} X_i}
\]
\end{lemma}

\begin{proof} We will prove that there is a natural \cstar-algebra isomorphism between $C(\Sigma_\cF X_i)$ and $C(\bigcap_{A\in \cF} \overline{\bigcup_{i\in A} X_i})$. 
	It will be convenient to write 
	\[
	\textstyle X_A=\overline{\bigcup_{i\in A} X_i}
	\]
	for $A\subseteq \bbI$. 
	Define $\Phi\colon \prod_i C(X_i)\to C(\beta \bigsqcup_i X_i)$ by letting $\Phi((f_i))$ be $\beta\sum_i f_i$, the unique continuous extension of the function equal to $f_i$ on $X_i$ to $\beta (\bigsqcup_i X_i)$. Then 
	$\limsup_{i\to \cF} \|f_i-g_i\|=0$ if and only if for every $\varepsilon>0$ the set $\{i: \|f_i-g_i\|\leq \varepsilon\}$ belongs to $\cF$. This is in turn equivalent to 
	\[
	(\forall \varepsilon>0)(\exists A\in \cF)\|(\Phi((f_i))-\Phi((g_i)))\rs X_A \|\leq \varepsilon, 
	\]
	which is the same as 
	$(\Phi((f_i))-\Phi((g_i)))\rs \bigcap_{A\in \cF} X_A$. 
	
	Since the \cstar-algebras are isomorphic, so are their Gelfand spectra. 
\end{proof}

 \L o\'s's Theorem holds for continuous logic (\cite[Theorem~16.2.8]{Fa:STCstar}), as does its analog for reduced products, the Feferman--Vaught theorem (\cite{ghasemi2016reduced}, \cite[Theorem~16.3.1]{Fa:STCstar}).

The \emph{density character} of a topological space  is the smallest cardinality of a dense subset.

\begin{proposition}\label{P.KeislerShelah}
For compact Hausdorff spaces $X$ and $Y$ the following are equivalent. 
\begin{enumerate}
	\item $X$ and $Y$ have the same co-theory. 
	\item $X$ and $Y$ have homeomorphic ultracopowers. 
\end{enumerate}
If $X$ and $Y$ are metrizable and CH holds, then the above are equivalent to the following. 
\begin{enumerate}[resume]
	\item Ultracopowers of $X$ and $Y$  associated with any  nonprincipal ultrafilter on $\bbN$ are homeomorphic. 
\end{enumerate}
\end{proposition}

\begin{proof}
	The equivalence of the first two statements is a special case of the Keisler--Shelah theorem (see \cite[Theorem 5.7]{BYBHU}, \cite[Theorem~6.8]{hart2023an}; also \cite[Notes to Chapter 16]{Fa:STCstar}). 
	
 The equivalence of the first two statements with the third under the assumption that $X$ and $Y$ are metrizable and CH holds follows from two facts. The first is that every ultrapower associated to a nonprincipal ultrafilter on~$\bbN$ is countably saturated (\cite[Theorem 16.4.1]{Fa:STCstar}), and therefore saturated if of density character $\aleph_1$. The second fact (due to Keisler  in the discrete case, and attributed to Keisler in general since the proof of the theorem in case of continuous logic is analogous to the proof in the discrete case) asserts that two elementarily equivalent saturated models of the same density character are isomorphic (\cite[Corollary 16.6.5]{Fa:STCstar}).  
\end{proof}

\subsection{Universal theories and embeddability}

Given a \cstar-algebra $C(X)$, a formula $\varphi(\bar x)$, and a tuple $\bar a$ in $C(X)$ of the same sort as $\bar x$, the interpretation $\varphi(\bar a)^{C(X)}$ is defined naturally by induction on the complexity of $\varphi$. For a sentence $\varphi$, some authors interpret the relation $\varphi^{C(X)}=0$ as `$\varphi$ is true in $C(X)$'. While this interpretation is somewhat contrived, the following closely related terminology makes some sense in light of Proposition~\ref{P.EmbeddingIntoUltrapower}. 

\begin{definition}\label{D.Universal}
	Formulas of the form, for $m\geq 1$,  
	\[
	\sup_{\|x(0)\|\leq m(0)}\sup_{\|x(1)\|\leq m(1)}\dots \sup_{\|x(n-1)\|\leq m(n-1)} \varphi,
	\] 
	where $\varphi$ is quantifier-free, are called \emph{universal}. If all free variables of $\varphi$ are included among $x(j)$, for $j<n$, then this is a universal sentence. 
\end{definition}

\begin{proposition}\label{P.EmbeddingIntoUltrapower}
	For compact Hausdorff spaces $X$ and $Y$ the following are equivalent. 
	\begin{enumerate}
		\item \label{1.Embedding}
		For every universal sentence $\varphi$, $\varphi^{C(X)}\leq \varphi^{C(Y)}$. 
		\item \label{2.Embedding} Some ultracopower of $Y$ maps onto $X$. 
		\item \label{3.Embedding} $C(X)$ embeds into some ultrapower of $C(Y)$. 
	\end{enumerate}
	If $X$ is compact and metrizable or even of weight $\leq\aleph_1$, then the above are equivalent to the following. 
	\begin{enumerate}[resume]
		\item \label{4.Embedding}  Some ultracopower of $Y$ associated with a nonprincipal ultrafilter on~$\bbN$ maps onto $X$. 
	\end{enumerate}
\end{proposition}

\begin{proof} \eqref{2.Embedding} and \eqref{3.Embedding} are equivalent by Theorem~\ref{T.GN.Duality}, and \eqref{3.Embedding} trivially implies \eqref{1.Embedding}. 
	\eqref{1.Embedding} implies \eqref{2.Embedding} by \cite[Theorem 2.3.5]{Muenster} (the latter is stated in terms of atomic diagrams).

	Now assume $X$ has weight $\leq \aleph_1$, hence by Lemma~\ref{L.MetrizableSeparable} $C(X)$ has density character $\leq\aleph_1$. Since 
	\eqref{4.Embedding} implies \eqref{2.Embedding} even without this assumption,  it suffices to prove that \eqref{3.Embedding} implies \eqref{4.Embedding}. By the former, $C(X)$ embeds into some ultrapower of $C(Y)$. 	
	Since $C(X)$ has density character $\leq \aleph_1$ by Lemma~\ref{L.MetrizableSeparable},   by the Downward L\"owenheim--Skolem theorem, $C(X)$ embeds into an elementary submodel $M$ of density character $\leq\aleph_1$ of this ultrapower. Fix a nonprincipal ultrafilter $\cU$ on $\bbN$. 
Then the ultrapower $C(Y)_\cU$ is countably saturated, $M$ embeds into it elementarily and therefore $C(X)$ embeds into $C(Y)_\cU$. By the duality, $X$ is a continuous image of the ultracopower of $Y$ associated with $\cU$.  
\end{proof}

\subsection{Types and saturation}
In the discrete logic, each formula is Boolean-valued, hence (once interpreted in some structure) each formula is true or false. In the continuous logic the values of formulas are real numbers, hence formally it does not make sense talking about a formula being `true' or `false' in a given structure.\footnote{Some authors  nevertheless interpret $\varphi^M=0$ as `$\varphi$ is true in $M$'. This convention makes some sense, but adopting it opens a big can of worms.} 
 Therefore we  consider a type as a set of conditions (see however the alternative view of a type as a functional, \cite[Appendix~C]{Fa:STCstar}). 
A \emph{(closed) condition} in an $n$-tuple of variables $\bar x=(x_0,\dots, x_{n-1})$ is an expression of one of the forms 
\[
\varphi(\bar x)\leq r, \qquad \varphi(\bar x)=r, \qquad\varphi(\bar x)\geq r. 
\]
We allow the formula $\varphi$ to be in the expanded language $\calL_{C(X)}$ (Definition~\ref{D.Expanded}).
Given a condition $\varphi(\bar x)\leq r$  in $\calL_{C(X)}$ and $\bar a$ in $C(X)$ of the appropriate sort, we say that 
`$\bar a$ satisfies $\varphi(\bar x)\leq r$' if  
$\varphi(\bar a)^{C(X)}\leq r$. We say that 
$\bar a$ $\varepsilon$-satisfies $\varphi(\bar x)\leq r$'  if $\varphi(\bar a)^{C(X)}\leq r+\varepsilon$. Satisfaction and $\varepsilon$-satisfaction of 
conditions of the form 
 $\varphi(\bar x)=r$ and $\varphi(\bar x)\geq r$ is defined analogously. 
\begin{definition}
A \emph{type over a \cstar-algebra $A$ in variables $\bar x$} is a set $\bt(\bar x)$ of conditions in $\calL_{A}$ in~$\bar x$. 
We say that $\bt(\bar x)$ is \emph{realized} in $A$ if some tuple $\bar a$ in $A$ satisfies all conditions in $\bt(\bar x)$. A type $\bt(\bar x)$ is \emph{approximately realized} in $A$ if for every finite subset $\bt_0$ of $\bt(\bar x)$ and every $\varepsilon>0$ some $\bar a$ in $A$ $\varepsilon$-satisfies each condition in $\bt_0$. 

We say that $A$ is \emph{countably saturated}\footnote{Common expression in model-theoretic literature is $\aleph_1$-saturated, and countable saturation is a considerably weaker property.} if every countable type over~$A$ that is approximately realized in $A$ is realized in $A$. It is \emph{saturated} if every type of cardinality smaller than the density character of $A$ that is approximately realized in $A$ is realized in $A$. 
\end{definition}

Many applications of types and saturation in continuous logic are given in \cite[\S 16.1]{Fa:STCstar} and in \cite{Muenster}.

For the following well known analog of even better known results in classical model theory see \cite[Theorem~16.4.1 and Theorem~16.5.1]{Fa:STCstar}. 

\begin{theorem} \label{T.Ultraproducs.Saturated} Suppose that $X_n$, for $n\in \bbN$, are compact Hausdorff spaces. If $\cU$ is a nonprincipal ultrafilter on $\bbN$, then the ultraproduct $\prod_\cU C(X_n)$ is countably saturated. 
	The reduced product $\prod_{\Fin} C(X_n)$ is countably saturated. 
	
	If in addition CH holds and all $X_n$ are metrizable, then 	
	$\prod_\cU C(X_n)$ and $\prod_{\Fin} C(X_n)$ are saturated.  \qed 
\end{theorem}

In the following, let (see Definition~\ref{D.Theory}---the theory is a functional---and note that the limit exists because for each sentence $\varphi$ the set of its possible values is a bounded subset of $\bbR$)
\[
\lim_{n\to \cU} \coTh(X_n)\text{ is the map } \varphi\mapsto \lim_{n\to \cU} \varphi^{C(X_n)},\text{ where $\varphi$ is a sentence}. 
\]
\L o\'s's theorem implies that this is the co-theory of $\sum_\cU X_n$. 

The limit $\lim_{n\to \infty} \coTh(X_n)$ is defined analogously, if $\lim_{n\to \infty} \varphi^{C(X_n)}$ exists for all $\varphi$. 

\begin{corollary}Suppose that $X_n$ and $Y_n$, for $n\in \bbN$, are infinite compact metrizable spaces  and $\cU$ is a nonprincipal ultrafilter on $\bbN$. 
	If CH holds then we have the following. 
	\begin{enumerate}
		\item \label{I1} Each one of $\sum_{\cU} X_n$ and $\sum_{\Fin} X_n$ has $2^{\aleph_1}$ autohomeomorphisms. 
		\item \label{I2} If $\lim_{n\to \cU} \coTh(X_n)=\lim_{n\to \cU} \coTh(Y_n)$ then $\sum_\cU X_n$ and $\sum_\cU Y_n$ are homeomorphic. 
\item \label{I3} If $\lim_{n\to \infty} \coTh(X_n)=\lim_{n\to \infty} \coTh(Y_n)$ then the \v Cech--Stone remainders of $\bigsqcup_n X_n$ and $\bigsqcup_n Y_n$ are homeomorphic. 
	\end{enumerate}
\end{corollary}

\begin{proof} By Theorem~\ref{T.Ultraproducs.Saturated} and CH, these ultraco-products and reduced co-products are saturated. By the continuous version of Keisler's theorem (\cite[Theorem~16.6.3]{Fa:STCstar}), every saturated model has the maximal number of automorphisms and \eqref{I1} follows. \eqref{I2} follows by \L o\'s's Theorem and 
Proposition~\ref{P.KeislerShelah}.  	\eqref{I2} 
	A corollary to the Feferman--Vaught theorem (\cite{ghasemi2016reduced}, \cite[Theorem~16.3.1]{Fa:STCstar}) implies that if $\lim_{n\to \infty} \coTh(X_n)=\lim_{n\to \infty} \coTh(Y_n)$, then (using Lemma~\ref{L.betaX}) we have that the spaces $\sum_{\Fin} X_n=\beta \bigsqcup_n X_n\setminus \bigsqcup_n X_n$ and $\sum_{\Fin} Y_n=\beta\bigsqcup_n Y_n\setminus Y_n$ are coelementarily equivalent and \eqref{I3} again follows by  Theorem~\ref{T.Ultraproducs.Saturated}. 
	\end{proof}

\subsection{Co-axiomatizable properties of compacta II}

By a standard result (\cite[Theorem~2.4.1]{Muenster}), a class of metric structures is axiomatizable if and only if it  is closed under isomorphisms, ultraproducts, and ultraroots (i.e., if an ultrapower of $A$ belongs to the class then so does $A$). 
In the paragraph preceding Theorem~5.5, a class of compacta is called closed if it is closed under co-ultraproducts and open if its complement is closed under co-ultraproducts.  Being both closed and open implies that both the class and its complement are axiomatizable, hence  \cite[Theorem~5.5]{bankstonsurvey} gives several additional co-axiomatizable classes of compacta, such as  being a (hereditarily) decomposable continuum and being a continuum of multicoherence degree $n$.

It would be nice to have explicit axiomatizations of these properties.

On the negative side, by \cite[Theorem 5.7]{bankstonsurvey}, being locally connected is not co-axiomatizable and \cite[Corollary~4.19]{eagle2024K-theory} implies that the vanishing first \v{C}ech cohomology is not preserved by elementary equivalence (and in particular is not axiomatizable). 

\section{Case study: Applications to continua}

The category of connected, compact Hausdorff spaces is dual to the (axiomatizable) category of unital, commutative, \cstar-algebras with no projections other than $0$ and $1$. 
The following is \cite[Lemma~2.1]{hart2005there}, with a proof recast using the Gelfand--Naimark duality,  proved at the end of this section,  instead of the language of the lattice of closed sets. 

\begin{proposition}\label{P.Continuum}
	For every compact Hausdorff space $X$, the following are equivalent. 
	\begin{enumerate}
		\item \label{1.P.Continuum} $X$ is connected. 
		\item \label{2.P.Continuum} An ultracopower of $[0,1]$ maps onto $X$. 
		\item \label{3.P.Continuum} $C(X)$ embeds into an ultrapower of $C([0,1])$. 
	\end{enumerate} 
		If $X$ is connected and has weight $\leq \aleph_1$, then $C(X)$ embeds into every ultrapower of $C([0,1])$ associated with a nonprincipal ultrafilter on $\bbN$. 
\end{proposition}

The proof will be given after two lemmas. 
Lemma~\ref{L.GromovHausdorf-convergence}, extracted from \cite{hart2005there}, is one of the key points of the proof of Proposition~\ref{P.Continuum}. 
Fix an ambient compact metric space $Z$ and fix a compatible metric on it; the Hilbert cube $[0,1]^\bbN$, usually denoted $Q$, equipped with the standard convex compatible metric will serve for our purpose. This metric has the property that every closed ball is a retract of $Q$. For $X\subseteq Z$ and $\varepsilon>0$ we consider the `$\varepsilon$-fattening of $X$' (the distance is computed with respect to the fixed compatible metric on $Z$): 
\[
(X)_\varepsilon=\{z\in Z: \dist(z,X)\leq \varepsilon\}. 
\]

\begin{lemma}\label{L.GromovHausdorf-convergence}
	Suppose that $X$ and $Y_n$, for $n\in \bbN$, are closed subspaces of a compact metric space $Z$, and that $X\subseteq Y_n\subseteq (X)_{1/n}$ for every $n$. Then~$X$ is a continuous image of the ultraco-product $\sum_\cU Y_n$ for any nonprincipal ultrafilter $\cU$. 
\end{lemma}

\begin{proof}Let $\cU$ be a nonprincipal ultrafilter on $\bbN$. 
	By Theorem~\ref{T.GN.Duality}, it will suffice to prove that $C(X)$ embeds into $\prod_\cU C(Y_n)$. 
	Let $a_n$, for $n\in \bbN$, be an enumeration of a dense unital $\bbQ+i\bbQ$-subalgebra $A$ closed under adjoints of $C(X)$. By the Tietze Extension Theorem, for each $n$ fix a continuous extension $\tilde a_n\in C(Z)$ of $a_n$ such that $\|a_n\|=\|\tilde a_n\|$, assuring that $\tilde 1=1$. 
	Let $\pi_\cU\colon \prod_n C(Y_n)\to \prod_\cU C(Y_n)$ be the quotient map. 
	Define $\Phi\colon A\to \prod_n C(Y_n)$ by 
	\[
	\Phi(a_m)=\langle \tilde a_m\rs Y_n: n\in \bbN\rangle. 
	\]
	We claim that $\Psi=\pi_\cU\circ\Phi$ is an isometric *-homomorphism of $A$ into $\prod_\cU C(Y_n)$.\footnote{A word of caution: $A$ is not a \cstar-algebra, hence *-homomorphisms from $A$ are not automatically continuous.}

For every $m$, the function $\tilde a_m$ is uniformly continuous, and therefore 
\[
\lim_n \sup_{y\in Y_n}|\tilde a_m(y)|=\sup_{x\in X} |a_m(x)|.
\] 
Thus $\lim_n \|\tilde a_m\rs Y_n\|=\|a_m\|$, and $\Psi$ is an isometry. 

To prove that it is a *-homomorphism, fix $m$ and $n$ and let $k$ be such that $a_k=a_m+a_n$. Then $b=\tilde a_k-\tilde a_m-\tilde a_n$ vanishes on $X$, and again by uniform continuity we have $\limsup_j \|b\rs Y_j\|=0$. Since this holds for all $m$ and $n$, $\Psi$ is additive. 

Analogous arguments show that for all $\lambda\in \bbQ+i\bbQ$ and all $m,n$ in $\bbN$, $\Psi(\lambda a_n)=\lambda \Psi(a_n)$, $\Psi(a_m a_n)=\Psi(a_m)\Psi(a_n)$, and $\Psi(a_m^*)=\Psi(a_m)^*$. Therefore $\Psi$ is an isometric *-homomorphism as claimed. 
The continuous extension of $\Psi$ to $C(X)$ is the required embedding into $\prod_\cU C(Y_n)$. 
\end{proof}

\begin{lemma}\label{L.Nondegenerate}
Every nondegenerate continuum $X$ has the same universal co-theory as $[0,1]$.
\end{lemma}	

\begin{proof}
	Since  $X$ is nondegenerate, it maps onto~$[0,1]$, and it remains to prove that an ultracopower of $[0,1]$ maps onto $X$. Since the conclusion of this lemma is concerned only with the co-theory of $X$, 
	by the L\"owenheim--Skolem theorem (\cite[Theorem~7.1.3]{Fa:STCstar}) and Theorem~\ref{T.GN.Duality}, we may assume that~$C(X)$ is separable.  By Lemma~\ref{L.MetrizableSeparable} we may assume that $X$ is metrizable and we can  identify it with a closed subspace of the Hilbert cube~$Q$. 
	
	Fix $n$. By compactness, $X$ can be covered by finitely many closed $1/n$-balls in $Q$ with centres in $X$, $X\subseteq \bigcup_{j<m(n)} B_{n,j}$.  Every closed ball is a retract of $Q$, and since $[0,1]$ maps onto $Q$, it is a continuous image of $[0,1]$. Since the centres of $B_{n,j}$ belong to $X$, the space $Y_n=\bigcup_{j<m(n)} B_{n,j}$ is connected. By induction on $m(n)$, one proves that $Y_n$ is a continuous image of $[0,1]$. In particular, $C(Y_n)$ embeds into $C([0,1])$.

	Since $X\subseteq Y_n\subseteq (X)_{1/n}$, Lemma~\ref{L.GromovHausdorf-convergence} implies that $C(X)$ embeds into an ultraproduct of $C(Y_n)$, and therefore into an ultrapower of $C([0,1])$. 
	This completes the proof. 
\end{proof}

\begin{proof}[Proof of Proposition~\ref{P.Continuum}] 
Fix a compact Hausdorff space $X$ and assume \eqref{1.P.Continuum}, that it is connected. Since \eqref{2.P.Continuum} and \eqref{3.P.Continuum} are trivial if $X$ has only one point, we may assume it is nondegenerate and therefore by Lemma~\ref{L.Nondegenerate} it has the same universal theory as $[0,1]$. 
By Proposition~\ref{P.EmbeddingIntoUltrapower},  \eqref{2.P.Continuum}  follows.  

By Theorem~\ref{T.GN.Duality} \eqref{2.P.Continuum} and \eqref{3.P.Continuum} are equivalent. 

By \L o\'s's theorem and axiomatizability of connectedness, every ultracopower of $[0,1]$ is connected and so are all of its continuous images. Therefore,  \eqref{2.P.Continuum} implies \eqref{1.P.Continuum}.

It remains to prove that if $X$ has weight $\leq \aleph_1$ then $C(X)$ embeds into every ultrapower of $C([0,1])$ associated with a nonprincipal ultrafilter on~$\bbN$. The assumption on $X$ implies that the density character of $C(X)$ is $\leq \aleph_1$. The proof is analogous to the proof of \eqref{4.Embedding} in Proposition~\ref{P.EmbeddingIntoUltrapower}. 
We already know that  $C(X)$ embeds into some ultrapower of $C([0,1])$. 	
By the Downward L\"owenheim--Skolem theorem, $C(X)$ embeds into an elementary submodel $M$ of this ultrapower of density character $\leq \aleph_1$. Fix a nonprincipal ultrafilter $\cU$ on $\bbN$. 
Then the ultrapower $C([0,1])_\cU$ is countably saturated, $M$ embeds into it elementarily and therefore $C(X)$ embeds into $C([0,1])_\cU$. By the duality, $X$ is a continuous image of the ultracopower of $[0,1]$ associated with $\cU$.  
\end{proof}

\section{\v Cech--Stone remainders}

The following is a very special case of a general concept dear to this author's heart (see \cite{farah2025corona} for the complete story). 

\begin{definition}\label{D.Trivial}
	If $X$ is locally compact and Hausdorff, then an autohomeomorphism $g$ of $\beta X\setminus X$ is \emph{trivial} if there are closed subsets $Y,Z$ of $X$ whose complements have compact closure and a homeomorphism $f\colon Y\to Z$ such that the unique continuous extension $\beta f\colon \beta Y\to \beta Z$
	whose restriction to $\beta Y\setminus Y$ (identified with $\beta X\setminus X$, because the complement of $Y$ has compact closure)  agrees with $g$.
One analogously defines trivial homeomorphisms between different \v Cech--Stone remainders. 
\end{definition}

\subsection{Continuum Hypothesis}
By \cite{Ru}, CH implies that $\beta\bbN\setminus \bbN$ has $2^{\aleph_1}$ nontrivial autohomeomorphisms, and by \cite{Sh:Proper} there is a forcing extension in which all autohomeomorphisms of $\beta\bbN\setminus \bbN$ are trivial. (A friendly sketch of this formidable proof can be found in \cite[\S 7.1]{farah2025corona}.) Both of these results had been proved by using Stone duality and proving the corresponding assertion about automorphisms of the Boolean algebra $\cPN/\Fin$.

By using modern methods, it is very easy to construct $2^{\aleph_1}$ nontrivial autohomeomorphisms of $\beta\bbN\setminus\bbN$. 
First, $\cPN/\Fin$ is  countably saturated, and therefore saturated, under CH. Every saturated model of cardinality~$\aleph_1$ has~$2^{\aleph_1}$ automorphisms. Apply Stone duality  to transfer each of these automorphisms to a homeomorphism of $\beta\bbN\setminus\bbN$. Since there are only $2^{\aleph_0}<2^{2^{\aleph_0}}=2^{\aleph_1}$ (using CH in the last equality) trivial automorphisms, the conclusion follows.  
(See \cite{dow2002applications} for two other proofs.)
The fact that the theory of atomless Boolean algebras admits elimination of quantifiers means that every injective homomorphism between atomless subalgebras of $\cPN/\Fin$ is automatically an elementary map, making the construction even more fool-proof. 
(Note that Rudin's original proof is different  and that the model-theoretic tools used in the proof just described were not available at the time. Rudin proved a stronger result,  that every two P-points of character~$\aleph_1$ are conjugate by an autohomeomorphism of $\beta\bbN\setminus \bbN$, regardless of whether CH holds or not.)

The model-theoretic approach also gives an easy proof of Parovi\v cenko's theorem which implies that all \v Cech--Stone remainders of locally compact, non-compact, zero-dimensional, separable metric spaces $X$ look the same (and are even homeomorphic under CH)  (\cite{dow2002applications}). 
 In \cite{eagle2015saturation} it was proved that for such $X$ the \cstar-algebra $C(\beta X\setminus X)$ is countably saturated and its theory admits elimination of quantifiers. The latter is analogous to, but not a consequence of, the well-known  fact that atomless Boolean algebras admit elimination of quantifiers; see \cite[Theorem~5.17]{eagle2015saturation}.

The situation is considerably more complicated if one drops zero-di\-men\-sio\-na\-li\-ty.

One could expect that analogous methods apply to \v Cech--Stone remainders of other locally compact, non-compact Polish spaces. This is true, modulo a few  hurdles.  
By \cite[Corollary 3.4]{eagle2015pseudoarc} together with the main result of~\cite{eagle2015quantifier}, the only theories of unital, commutative \cstar-algebras that admit quantifier elimination are the theories of $\bbC$, $\bbC^2$, and $C(2^{\bbN})$,\footnote{This is assuming that the language has a constant for the unit; the non-unital case is equally restrictive.} where $2^{\bbN}$ denotes the Cantor space. The failure of quantifier elimination does not present a real problem, but there is an issue with saturation.

In the following, $\bigsqcup_n X_n$ denotes the discrete sum of spaces $X_n$, for $n\in \bbN$. 

\begin{proposition}
	Suppose that $X_n$, for $n\in \bbN$, are compact Hausdorff spaces and $X=\bigsqcup_n X_n$. Then $C(\beta X\setminus X)$ is countably saturated. In particular, if in addition all $X_n$ are metrizable (or of weight $\leq \aleph_1$) then  CH implies that it has $2^{\aleph_1}$ automorphisms. 
\end{proposition}

\begin{proof}
	We have that\footnote{The direct sum $\bigoplus_n C(X_n)$ is taken in the category of \cstar-algebras, and it is therefore the completion of the algebraic direct sum.}  $C_0(X)=\bigoplus_n C(X_n)$, and that $C_b(X)$ can be identified with 
	$\prod_n C(X_n)$.\footnote{Remember, this is not the Cartesian product.}
The quotient $\prod_n C(X_n)/\bigoplus_n C(X_n)$ is, by \cite[Theorem~1.5]{FaSh:Rigidity} (see however a simpler proof in \cite[\S 16.5]{Fa:STCstar}), countably saturated. 

Assume CH and that all $X_n$ have weight $\leq \aleph_1$.  Then each $C(X_n)$ has density character $\leq\aleph_1$, and so does the reduced product. It is therefore fully saturated and has $2^{\aleph_1}$ automorphisms (\cite[Proposition~16.7.1 (1)]{Fa:STCstar}). 
\end{proof}

The following result of Saeed Ghasemi (\cite{ghasemi2016reduced})  will be used in the proof of Corollary~\ref{C.Independent} below. 

\begin{theorem}[Ghasemi's trick]\label{T.Ghasemi}
		Suppose that $X_n$, for $n\in \bbN$, is a sequence of compact metric spaces. Then there is an infinite $I\subseteq \bbN$ such that for every infinite $J\subseteq I$, CH implies that $X_I=\bigsqcup_{n\in I}X_n$ and $X_J=\bigsqcup_{n\in J} X_n$ have homeomorphic \v Cech--Stone remainders.  
\end{theorem}

The choice of $I$ in this theorem does not depend on CH.

\begin{proof} Since the language of \cstar-algebras is separable (see Remark~\ref{R.Separability}), the space of all sentences is separable, and we can choose an infinite $I\subseteq \bbN$ such that for every sentence $\varphi$ the limit $\lim_{n\in I, n\to \infty}\varphi^{C(X_n)}$ exists. This limit clearly does not change when passing to a further infinite $J\subseteq I$. 
	By a corollary to Ghasemi's Feferman--Vaught theorem (\cite{ghasemi2016reduced}, \cite[Theorem~16.6.3]{Fa:STCstar}), the \cstar-algebras 
	\begin{align*}
	C(\beta X_I\setminus X_I)&=\textstyle\prod_{n\in I} C(X_n)/\sum_{n\in I}C(X_n) \text{ and }\\ 
	C(\beta X_J\setminus X_J)&=\textstyle\prod_{n\in J} C(X_n)/\sum_{n\in J}C(X_n)
	\end{align*}
	 are elementarily equivalent. Since they are both countably saturated, CH implies that they are isomorphic, and by the Gelfand--Naimark duality $X_I$ and~$X_J$ have homeomorphic remainders.  
\end{proof}

The \v Cech--Stone remainder of the half-line $\bbH=[0,\infty)$ is a well-studied continuum (\cite{Hart:Cech}). 
The Continuum Hypothesis implies that $\beta \bbH\setminus \bbH$ has $2^{\aleph_1}$ nontrivial autohomeomorphisms (\cite[\S 9]{Hart:Cech}, \cite{yu1991automorphism}). 

\begin{theorem}\label{T.Saturated.betaH-H} The \cstar-algebra $C(\beta \bbH\setminus \bbH)$ is countably saturated. In particular, CH implies that it has $2^{\aleph_1}$ nontrivial autohomeomorphisms. 
	\end{theorem}

\begin{proof} Let $\partial X$ denote the topological boundary of $X$. 
Countable saturation of $C(\beta\bbH\setminus \bbH)$ follows by \cite[Theorem~2.5]{FaSh:Rigidity}: take $X_n=[0,n]$; then $\bbH=\bigcup_n X_n$ and $\partial X_n$ is a singleton for all $n$. As before, CH implies that $C(\beta \bbH\setminus \bbH)$, as a saturated model, has $2^{\aleph_1}$ automorphisms (\cite[Proposition~16.7.1 (1)]{Fa:STCstar}) and the conclusion follows by Theorem~\ref{T.GN.Duality}. 
\end{proof}

Countable saturation can get us slightly further. 
By \cite[Theorem~2.5]{FaSh:Rigidity}, if~$X$ is a locally compact, non-compact Polish space that can be presented as a union of a  strictly increasing sequence of compact subspaces $X_n$ such that  $\sup_n |\partial X_n|<\infty$, then $C(\beta X\setminus X)$ is countably saturated. \cite[Theorem~2.2]{FaSh:Rigidity} suggests that this condition on $X$ may be necessary for countable saturation of $C(\beta X\setminus X)$.

Is there a pair of connected (locally compact, non-compact, Polish) spaces $X$ and $Y$ such that the assertion that $\beta X\setminus X$ and $\beta Y\setminus Y$ are homeomorphic is independent from ZFC? In order to have these remainders homeomorphic for nontrivial reason, one would have to prove that the associated \cstar-algebras are elementarily equivalent, and at present we do not have an efficient way of computing the theory of a connected space. One route to this would be via an extension of the Feferman--Vaught theorem;  such an extension would be of great interest in its own right. 

If $X$ and $Y$ satisfy the assumptions of \cite[Theorem~2.5]{FaSh:Rigidity}, then 
a simpler route could use a combination of Ghasemi's trick with the methods of \cite[\S 3]{dow1993cech}, where the \v Cech--Stone remainder of a continuum $X$ is considered as a continuous image of the \v Cech--Stone remainder of the space $[0,1]\times \bbN$ (known as $\mathbb M$).

\subsection{Nontrivial autohomeomorphisms without countable saturation.}
If $X$ can be presented as an increasing union of compact subspaces, then $C(\beta X\setminus X)$ satisfies a weak form of countable saturation called \emph{countable degree-1 saturation} (see \cite[\S 15]{Fa:STCstar}). While fairly useful for analyzing properties of \cstar-algebras, this property does not provide a sufficient amount of interesting information about \v Cech--Stone remainders to justify introducing it here. 

The following inconvenient fact is \cite[Corollary~2]{farah2023obstructions} (see also Theorem 1 of the same paper for more bad news). 
 
\begin{theorem}
If $X$ is locally compact Hausdorff and  $\bbR^2$ embeds as a closed subspace into $X$, then $C(\beta X\setminus X)$ is not countably saturated. \qed
\end{theorem}

 Nevertheless, \cstar-algebras provide a good vantage point for analyzing autohomeomorphisms of \v Cech--Stone remainders of $\bbR^n$ for $n\geq 2$ and other locally compact, non-compact manifolds. 
The following is \cite[Theorem~1]{vignati2017nontrivial}. 

\begin{theorem}
If $X$ is a locally compact, metrizable, non-compact manifold, then CH implies that $\beta X\setminus X$ has $2^{\aleph_1}$ autohomeomorphisms. \qed 
\end{theorem}  

The proof uses a clever stratification of \cstar-algebras associated to coronas of the so-called `flexible' spaces (see \cite[Definition~2.1]{vignati2017nontrivial}). 

\subsection{Forcing axioms}
Considerably more interesting (in the author's opinion) and deeper (absolutely) than the trivializing effect of CH on \v Cech--Stone remainders and other quotients is the fact that under forcing axioms, coronas of locally compact, non-compact, Polish spaces have only trivial autohomeomorphisms. 
In the case of $\bbN$ this is a result of Shelah (\cite{Sh:Proper}) asserting that all automorphisms of $\cPN/\Fin$ are (consistently with ZFC) trivial, as adapted to the context of forcing axioms  (in decreasing order of strength) in \cite{ShSte:PFA, Ve:OCA, de2023trivial}.

The following ultimate commutative generalization of these results is taken from \cite[Theorem~C]{vignati2022rigidity} ($\OCAT$ is Todorcevic's version of the Open Colouring Axiom, \cite{To:Partition}, and $\MA$ is Martin's Axiom, \cite{Ku:Set}; see Definition~\ref{D.Trivial} for (non)trivial autohomeomorphisms). 

\begin{theorem}\label{T.Vignati.FA}
	$\OCAT$ and $\MA$ together imply that for all second-countable, locally compact, non-compact spaces $X$ and $Y$, every homeomorphism between $\beta X\setminus X$ and $\beta Y\setminus Y$ is trivial. 
	 \qed 
\end{theorem}

If $X$ and $Y$ are zero-dimensional, then the conclusion  of Theorem~\ref{T.Vignati.FA} follows already from $\OCAT$  (see \cite[Corollary 11.1.2]{farah2026quotient}, where other proofs from $\OCAT$ alone of quotient rigidity theorems originally proved using $\OCAT$ and $\MA$ can be found). 

In addition to forcing axioms, the proof of Theorem~\ref{T.Vignati.FA} relies on Ulam-stability of approximate *-homomorphisms between commutative \cstar-al\-geb\-ras proved in \cite{semrl1999non}.

To give an illustration of these results, for an infinite $I\subseteq \bbN$ let 
\[
\textstyle Z_I=\bigsqcup_{n\in I} [0,1]^n. 
\]

\begin{corollary}\label{C.Independent}
	There is an infinite set $I\subseteq \bbN$ such that for every infinite $J\subseteq I$ such that $I\setminus J$ is infinite the assertion that $Z_I$ and $Z_J$ have homeomorphic \v Cech--Stone remainders is independent from ZFC. 
\end{corollary}

\begin{proof}
	By Ghasemi's trick (Theorem~\ref{T.Ghasemi}), there exists an infinite $I\subseteq \bbN$ such that for every infinite $J\subseteq I$, CH implies that $\beta Z_I\setminus Z_I$ and $\beta Z_J\setminus Z_J$ are homeomorphic. 
	By Theorem~\ref{T.Vignati.FA} and elementary dimension theory, $\OCAT$ and $\MA$ imply that these spaces are homeomorphic if and only if $I\setminus J$ is finite. 
\end{proof}

See \cite{farah2025corona} for more on rigidity results along these lines. 
$\OCAT$ is known to suffice for the analogous conclusion in the case of the Calkin algebra (\cite{Fa:All}, \cite[\S 17]{Fa:STCstar}). 

\subsection{ZFC}

The following is the main result of \cite{DoHa:Universal}. 

\begin{theorem}\label{T.UniversalContinuum}
Every continuum $X$ of weight not greater than $\aleph_1$ is a continuous image of $\beta \bbH\setminus \bbH$.
\end{theorem}  

\begin{proof} We may assume that $X$ is nondegenerate. 
By Lemma~\ref{L.Nondegenerate}, $C(X)$ and $C(\beta\bbH\setminus \bbH)$ have the same universal theory.  
	By Proposition~\ref{P.EmbeddingIntoUltrapower}, $C(X)$ embeds into  $C(\beta\bbH\setminus \bbH)_\cU$ for some ultrafilter~$\cU$. By the Downwards L\"owenheim--Skolem theorem (\cite[Theorem~7.1.3]{Fa:STCstar}), the image of $C(X)$  is included in an elementary submodel $M$ of $C(\beta\bbH\setminus \bbH)_\cU$  that has density character $\aleph_1$. Since $C(\beta\bbH\setminus \bbH)$ is countably saturated by Theorem~\ref{T.Saturated.betaH-H},   it is $\aleph_1$-universal (i.e., every model of density character $\leq\aleph_1$ elementarily equivalent to it embeds into it elementarily), hence $M$ embeds into it and therefore so does $C(X)$. 
\end{proof}

By the L\"owenheim--Skolem theorem (\cite[Theorem~7.1.3]{Fa:STCstar}), there is a compact metrizable space $X$ such that $C(X)$ is elementarily equivalent to $C(\beta\bbH\setminus \bbH)$. 
If CH holds, then $C(\beta\bbH\setminus\bbH)$ is isomorphic to the ultrapower of $C(X)$ associated with any nonprincipal ultrafilter $\cU$ on $\bbN$. One could ask whether there is a canonical example of a compact metrizable space $X$ such that $C(X)$ is elementarily equivalent to $C(\beta\bbH\setminus \bbH)$; however, it is not clear whether the theory of the latter is decidable in ZFC. 
In \cite[Theorem 3.2]{dow1993cech} it was proved that certain subcontiuum of $\beta \bbH\setminus \bbH$ is coelementarily equivalent (and homeomorphic if CH holds) to it.

\section{Reflection}\label{S.Reflection}

Elementary submodels are routinely used to prove reflection results, often in conjunction with nontrivial set-theoretic assumptions. 
With Gelfand--Naimark duality this is possible, but rather subtle for more than one reason. 

Before delving into this section, the reader should be warned that results along similar lines  had been obtained without using \cstar-algebras in for example  \cite{bandlow1991characterization}, \cite{bandlow1991construction}, 
\cite{eisworth2006elementary}, \cite{dow1992set}, 
\cite{kunen2003compact}. Nevertheless, it is (in this author's opinion) likely that the \cstar-algebraic vantage point will yield additional applications.

I will assume that the reader is familiar with the structure $H_\theta$ of all sets whose transitive closures have cardinality $<\theta$ (see e.g., \cite[Appendix~A]{Fa:STCstar}). 
Suppose that $X$ is a compact Hausdorff space and $\theta$ is a sufficiently large regular cardinal ($\theta>2^{2^{|X|}}$ should suffice). 
Suppose that $M\prec H_\theta$ is an elementary submodel such that $X\in M$. 
Unless $M$ is closed under $\omega$-sequences, $M\cap C(X)$ need not be norm-closed unless $M$ is closed under countable sequences, and if $M$ is countable it is not even a $\bbC$-algebra. This is remedied by taking the norm-closure of $C(X)\cap M$. 
By Theorem~\ref{T.GN.Duality}, the Gelfand spectrum of $\overline{C(X)\cap M}$, denoted $X^M$, is a continuous image of~$X$.  
One should note that $X^M$ need not be a subset of $M$ (and even $\overline{C(X)\cap M}$ need not be a subset of $M$, unless $M$ is closed under $\omega$-sequences). 
See \cite[Lemma 2.5]{farah2021corson} for a more detailed analysis of the relation between $X$ and $X^M$. 

On the brighter side, the set
\[
\{\overline{C(X)\cap M}: M\prec H_\theta, X\in M, M\text{ countable}\}
\]
is a club (closed unbounded set)\footnote{Warning: This `club' is not closed under unions of countable increasing sequences. It is instead closed under taking suprema of countable increasing sequences, hence a club in the sense of \cite[Definition~6.2.6]{Fa:STCstar}.} in the poset $\Sep(C(X))$ of separable substructures of $C(X)$. Such clubs are reasonably well-behaved (see \cite[Chapters 6 and 7]{Fa:STCstar}), and stationary subsets of $\Sep(C(X))$ even satisfy a continuous analog of the Pressing Down Lemma (\cite[Proposition~6.5.6]{Fa:STCstar}). 
One should however remember that it is the continuous functions on $X$, rather than the points of $X$, that the elementary submodels of $H_\theta$ can relate to.

Using these methods, a reflection principle for Corson compacta was proved in \cite[Theorem 1]{farah2021corson} by combining Gelfand--Naimark duality with set-theoretic reflection methods. These results had been proven earlier without appealing to the Gelfand--Naimark duality.

The following is \cite[Proposition 2.6]{farah2021corson}. 

\begin{proposition}
Suppose that $X$ is a compact Hausdorff space such that every continuous image of $X$ of weight at most $2^{\aleph_0}$ is Fr\'echet. Then $X$ is Fr\'echet. \qed
\end{proposition}

A similar reflection principle for uncountable tightness was proved, assuming CH, in \cite{juhasz1992convergent}. 
For more information on the unruly behaviour of elementary submodels in this context, see the examples given in \cite[\S 4]{farah2021corson}. 

\section{An open-ended concluding section}
The \cstar-algebra $C(2^\bbN)$ is injectively universal in the class of separable, unital, commutative, \cstar-algebras; this is just a reformulation of the familiar fact that the Cantor space is surjectively universal in the class of compact metrizable spaces, using Lemma~\ref{L.MetrizableSeparable}. 
On the other hand, there is no separable unital abelian \cstar-algebra without nontrivial projections that is injectively universal for the class of separable unital abelian \cstar-algebras without nontrivial projections 
(the dual statement is proved in \cite{waraszkiewicz1934probleme}). 

One often needs to analyze not just a single compact Hausdorff space but a continuous function  $f\colon X\to Y$ between such spaces.  For such analysis using elementary submodels of $H_\theta$ (as in \S\ref{S.Reflection}) see e.g., \cite{eisworth2006elementary} and  \cite{bandlow1991construction}. By Theorem~\ref{T.GN.Duality}, such triple $f\colon X\to Y$ corresponds to a triple $(C(X), C(Y), f_*)$ of \cstar-algebras and a *-homomorphism between them. The class of such triples is easily shown to be axiomatizable. 

One interesting situation is the following. Using the notation from \S\ref{S.ultracoproducts}, given compact metrizable spaces $X_n$, for $n\in \bbN$, and a nonprincipal ultrafilter $\cU$ on $\bbN$, the space $\sum_\cU X_n$ is a subspace of $\sum_{\Fin} X_n$ (where $\Fin$ is the Fr\'echet filter). Equivalently, $\prod_\cU C(X_n)$ is a quotient of $\prod_{\Fin} C(X_n)$. Theorem~\ref{T.Between} below is (modulo Theorem~\ref{T.GN.Duality}) a special case of \cite[Theorems~E and C]{farah2019between}, minus the notation ($2^{\bbN}$ stands for the Cantor space).

Recall that $\beta (X\times \bbN)\setminus (X\times \bbN)$ is homeomorphic to $\sum_{\Fin} X$ (Lemma~\ref{L.betaX}). 

\begin{theorem} \label{T.Between} Suppose that $X$ is a compact Hausdorff space and $\cU$ is a nonprincipal ultrafilter on $\bbN$. Then the following holds. 
	\begin{enumerate}
		\item \label{1.Between}  $\sum_{\Fin} X$   is coelementarily equivalent to $\sum_\cU (X\times 2^{\bbN})$. 
		
		\item \label{2.Between} If CH holds and $X$ is metrizable, then $\sum_{\Fin} X$   and $\sum_\cU (X\times 2^{\bbN})$ are homeomorphic. 
		
		\item \label{3.Between} If CH holds and $\cU$ is a P-point,   then   the canonical copy of $\sum_\cU X$ inside $\sum_{\Fin} X$   is a retract for every metrizable $X$. 
	\end{enumerate}   
\end{theorem}

\begin{proof} In\footnote{As the referee pointed out, this proof is not very informative. An interested reader is invited to read the (nontrivial) relevant parts of \cite{farah2019between}.} \cite{farah2019between} the notation $C(X)^\infty$, $C(X)^\cU$ is used for $\prod_{\Fin} C(X)$ and $\prod_\cU C(X)$, respectively.

	In \cite[\S 3]{farah2019between}, the functor $K$ is defined by $K(C(X))=C(X\otimes 2^{\bbN})$.  Thus~\cite[Proposition~3.5]{farah2019between} implies \eqref{1.Between} and \cite[Theorem~E]{farah2019between} gives \eqref{2.Between}. 
	Finally, there is a *-homomorphism $\Psi\colon \prod_\cU C(X)\to \prod_{\Fin} C(X)$   such that the quotient map $\pi_\cU\colon \prod_{\Fin} C(X)\to \prod_\cU C(X)$ satisfies $\pi_\cU\circ \Psi=\id_{\prod_{\cU} C(X)}$ (\cite[Theorem~C]{farah2019between}).  This, together with Theorem~\ref{T.GN.Duality},  gives \eqref{3.Between}. 
\end{proof}

By \cite[Theorem~C]{farah2019between}, if $\cU$ is not a P-point then the analog of 
\eqref{3.Between}  fails for some separable (not necessarily commutative) \cstar-algebra.

\bibliographystyle{plain}
\bibliography{gn-bib}

\begin{thebibliography}{10}

\bibitem{balogh1996small}
Z.T. Balogh.
\newblock A small {D}owker space in {ZFC}.
\newblock {\em Proc. Amer. Math. Soc.}, 124(8):2555--2560, 1996.

\bibitem{bandlow1991characterization}
I.~Bandlow.
\newblock A characterization of {C}orson-compact spaces.
\newblock {\em Comment. Math. Univ. Carolin}, 32(3):545--550, 1991.

\bibitem{bandlow1991construction}
I.~Bandlow.
\newblock A construction in set-theoretic topology by means of elementary
  substructures.
\newblock {\em Z. Math. Logik Grundlag. Math.}, 37(5):467--480, 1991.

\bibitem{Bankston1987reduced}
P.~Bankston.
\newblock {Reduced coproducts of compact Hausdorff spaces}.
\newblock {\em J. Symb. Log.}, 52(2):404--424, jun 1987.

\bibitem{bankston1997co}
P.~Bankston.
\newblock Co-elementary equivalence, co-elementary maps, and generalized arcs.
\newblock {\em Proc. Amer. Math. Soc.}, 125(12):3715--3720, 1997.

\bibitem{bankston2000some}
P.~Bankston.
\newblock Some applications of the ultrapower theorem to the theory of
  compacta.
\newblock {\em Applied Categorical Structures}, 8(1):45--66, 2000.

\bibitem{bankstonsurvey}
P.~Bankston.
\newblock A survey of ultraproduct constructions in general topology.
\newblock {\em Topology Atlas Invited Contributions}, (2):1--32, 2003.

\bibitem{bartovsovaa2011lelek}
D.~Barto{\v{s}}ov{\'a}, K.P. Hart, L.C. Hoehn, and B.~van~der Steeg.
\newblock Lelek's problem is not a metric problem.
\newblock {\em Top. Appl.}, 158:2479--2484, 2011.

\bibitem{BYBHU}
I.~Ben~Yaacov, A.~Berenstein, C.W. Henson, and A.~Usvyatsov.
\newblock Model theory for metric structures.
\newblock In Z.~Chatzidakis et~al., editors, {\em Model Theory with
  Applications to Algebra and Analysis, Vol. II}, number 350 in London Math.
  Soc. Lecture Notes Series, pages 315--427. London Math. Soc., 2008.

\bibitem{Black:Operator}
B.~Blackadar.
\newblock {\em Operator algebras}, volume 122 of {\em Encyclopaedia of
  Mathematical Sciences}.
\newblock Springer-Verlag, Berlin, 2006.
\newblock Theory of \cstar-algebras and von Neumann algebras, Operator Algebras
  and Non-commutative Geometry, III.

\bibitem{BrOz:C*}
N.~Brown and N.~Ozawa.
\newblock {\em {\cstar}-algebras and finite-dimensional approximations},
  volume~88 of {\em Graduate Studies in Mathematics}.
\newblock American Mathematical Society, Providence, RI, 2008.

\bibitem{chang1966continuous}
C.C. Chang and H.J. Keisler.
\newblock {\em Continuous model theory}.
\newblock Princeton: Princeton University Press, 1966.

\bibitem{de2023trivial}
B.~De~Bondt, I.~Farah, and A.~Vignati.
\newblock Trivial automorphisms of reduced products.
\newblock {\em Israel J. of Math.}, to appear, arXiv preprint 2307.06731.

\bibitem{dow1988introduction}
A.~Dow.
\newblock An introduction to applications of elementary submodels to topology.
\newblock {\em Topology Proc.}, 13(1):17--72, 1988.

\bibitem{dow1992set}
A.~Dow.
\newblock Set theory in topology.
\newblock {\em Recent progress in general topology (Prague, 1991);
  North-Holland, Amsterdam}, pages 167--197, 1992.

\bibitem{dow1993cech}
A.~Dow and K.P. Hart.
\newblock {\v{C}}ech--{S}tone remainders of spaces that look like $[0,\infty]$.
\newblock {\em Acta Univ. Carolinae. Mathematica et Physica}, 34(2):31--39,
  1993.

\bibitem{DoHa:Universal}
A.~Dow and K.P. Hart.
\newblock A universal continuum of weight {$\aleph$}.
\newblock {\em Trans. Amer. Math. Soc.}, 353(5):1819--1838, 2001.

\bibitem{dow2002applications}
A.~Dow and K.P. Hart.
\newblock Applications of another characterization of {$\beta \bbN\setminus
  \bbN$}.
\newblock {\em Top. Appl.}, 122(1-2):105--133, 2002.

\bibitem{eagle2015quantifier}
C.J. Eagle, I.~Farah, E.~Kirchberg, and A.~Vignati.
\newblock Quantifier elimination in \cstar-algebras.
\newblock {\em International Mathematics Research Notices},
  2017(24):7580--7606, 2017.

\bibitem{eagle2015pseudoarc}
C.J. Eagle, I.~Goldbring, and A.~Vignati.
\newblock The pseudoarc is a co-existentially closed continuum.
\newblock {\em Top. Appl.}, 207:1--9, 2016.

\bibitem{eagle2024K-theory}
C.J. Eagle and J.~Lau.
\newblock {{\(K\)}}-theory of co-existentially closed continua.
\newblock {\em J. Log. Anal.}, 16:22, 2024.
\newblock Id/No 1.

\bibitem{eagle2015saturation}
C.J. Eagle and A.~Vignati.
\newblock Saturation and elementary equivalence of \cstar-algebras.
\newblock {\em J. Funct. Anal.}, 269(8):2631--2664, 2015.

\bibitem{eisworth2006elementary}
T.~Eisworth.
\newblock Elementary submodels and separable monotonically normal compacta.
\newblock In {\em Proceedings of the 20th {S}ummer {C}onference on {T}opology
  and its {A}pplications}, volume~30, pages 431--443, 2006.

\bibitem{Fa:AQ}
I.~Farah.
\newblock {\em Analytic quotients: theory of liftings for quotients over
  analytic ideals on the integers}, volume 148 of {\em Memoirs Amer. Math.
  Soc.}
\newblock Amer. Math. Soc., 2000.

\bibitem{Fa:All}
I.~Farah.
\newblock All automorphisms of the {C}alkin algebra are inner.
\newblock {\em Ann. of Math. (2)}, 173:619--661, 2011.

\bibitem{farah2019between}
I.~Farah.
\newblock Between reduced powers and ultrapowers.
\newblock {\em JEMS}, 25(11):4369--4394, 2022.

\bibitem{Fa:STCstar}
I.~Farah.
\newblock {\em Combinatorial Set Theory of \cstar-algebras}.
\newblock Springer Monographs in Mathematics. Springer, second edition, 2026.

\bibitem{farah2026quotient}
I.~Farah.
\newblock {\em Quotient rigidity: Analytic quotients and \v{C}ech--{S}tone
  remainders}.
\newblock Math. Surveys and Monographs. Amer. Math. Soc., to appear, draft
  available upon request.

\bibitem{farah2025corona}
I.~Farah, S.~Ghasemi, A.~Vaccaro, and A.~Vignati.
\newblock Corona rigidity.
\newblock {\em Bull. Symb. Logic}, 31(2):195--287, 2025.

\bibitem{Muenster}
I.~Farah, B.~Hart, M.~Lupini, L.~Robert, A.~Tikuisis, A.~Vignati, and
  W.~Winter.
\newblock Model theory of \cstar-algebras.
\newblock {\em Memoirs of the Amer. Math. Soc.}, 271(1324), 2021.

\bibitem{farah2021corson}
I.~Farah and M.~Magidor.
\newblock Corson reflections.
\newblock {\em Annals of Pure and Applied Logic}, 172(5):102908, 2021.

\bibitem{farah2016axiomatizability}
I.~Farah and M.~R{\o}rdam.
\newblock Axiomatizability of the stable rank of \cstar-algebras.
\newblock {\em M\"unster J. Math.}, 10(2):269--275, 2017.

\bibitem{FaSh:Rigidity}
I.~Farah and S.~Shelah.
\newblock Rigidity of continuous quotients.
\newblock {\em J. Math. Inst. Jussieu}, 15(01):1--28, 2016.

\bibitem{farah2023obstructions}
I.~Farah and A.~Vignati.
\newblock Obstructions to countable saturation in corona algebras.
\newblock {\em Proc. Amer. Math. Soc.}, 151(03):1285--1300, 2023.

\bibitem{gelfand1939on}
I.~Gelfand and A.~Kolmogoroff.
\newblock On rings of continuous functions on topological spaces.
\newblock {\em C. R. (Dokl.) Acad. Sci. URSS, n. Ser.}, 22:11--15, 1939.

\bibitem{gelfand1943imbedding}
I.~Gelfand and M.~Neumark.
\newblock On the imbedding of normed rings into the ring of operators in
  {H}ilbert space.
\newblock {\em Matemati\v ceskii Zbornik}, 12(2):197--217, 1943.

\bibitem{geschke2002applications}
S.~Geschke.
\newblock Applications of elementary submodels in general topology.
\newblock {\em Synthese}, 133(1-2):31--41, 2002.

\bibitem{ghasemi2016reduced}
S.~Ghasemi.
\newblock Reduced products of metric structures: a metric {F}eferman--{V}aught
  theorem.
\newblock {\em J. Symb. Log.}, 81(3):856--875, 2016.

\bibitem{gillmanrings}
L.~Gillman and M.~Jerison.
\newblock {\em Rings of Continuous Functions}.
\newblock Springer, 1960.

\bibitem{hart2023an}
B.~Hart.
\newblock An introduction to continuous model theory.
\newblock In {\em Model theory of operator algebras}, pages 83--131. Berlin: De
  Gruyter, 2023.

\bibitem{Hart:Cech}
K.P. Hart.
\newblock The \v {C}ech-{S}tone compactification of the real line.
\newblock In {\em Recent progress in general topology ({P}rague, 1991)}, pages
  317--352. North-Holland, Amsterdam, 1992.

\bibitem{hart2005there}
K.P. Hart.
\newblock There is no categorical metric continuum.
\newblock In {\em Continuum theory: in honor of Professor David P. Bellamy on
  the occasion of his 60th birthday,}, volume~19 of {\em Aportaciones Mat.
  Investig.} Soc. Mat. Mexicana,, Mexico, 2007, arXiv preprint math/0509099.

\bibitem{henson1977first}
C.W. Henson, C.G. Jockusch~Jr, L.A. Rubel, and G.~Takeuti.
\newblock First order topology.
\newblock {\em Diss. Math.}, CXLIII, 1977.

\bibitem{juhasz1992convergent}
I.~Juhasz and Z.~Szentmiklossy.
\newblock Convergent free sequences in compact spaces.
\newblock {\em Proc. Amer. Math. Soc.}, 116(4):1153--1160, 1992.

\bibitem{kunen2003compact}
K.~Kunen.
\newblock Compact spaces, compact cardinals, and elementary submodels.
\newblock {\em Topology and its Applications}, 130(2):99--109, 2003.

\bibitem{Ku:Set}
K.~Kunen.
\newblock {\em Set theory}, volume~34 of {\em Studies in Logic (London)}.
\newblock College Publications, London, 2011.

\bibitem{Murphy:C*}
G.J. Murphy.
\newblock {\em \cstar-algebras and operator theory}.
\newblock Academic Press Inc., Boston, MA, 1990.

\bibitem{murray1943rings}
F.J. Murray and J.~von Neumann.
\newblock On rings of operators. {IV}.
\newblock {\em Ann. Math.}, 44(4):716--808, 1943.

\bibitem{Pede:Analysis}
G.K. Pedersen.
\newblock {\em Analysis now}, volume 118 of {\em Graduate Texts in
  Mathematics}.
\newblock Springer-Verlag, New York, 1989.

\bibitem{Ru}
W.~Rudin.
\newblock Homogeneity problems in the theory of \v{C}ech compactifications.
\newblock {\em Duke Mathematics Journal}, 23:409--419, 1956.

\bibitem{semrl1999non}
P.~{\v S}emrl.
\newblock Non-linear perturbations of homomorphisms on {$C (X)$}.
\newblock {\em Quarterly J. Math.}, 50(197):87--109, 1999.

\bibitem{Sh:Proper}
S.~Shelah.
\newblock {\em Proper Forcing}.
\newblock Lecture Notes in Mathematics 940. Springer, 1982.

\bibitem{ShSte:PFA}
S.~Shelah and J.~Stepr{\=a}ns.
\newblock {PFA} implies all automorphisms are trivial.
\newblock {\em Proceedings of the American Mathematical Society},
  104:1220--1225, 1988.

\bibitem{She:Notes}
D.~Sherman.
\newblock Notes on automorphisms of ultrapowers of {II}$_1$ factors.
\newblock {\em Studia Math.}, 195:201--217, 2009.

\bibitem{soukup2018infinite}
D.T. Soukup and L.~Soukup.
\newblock Infinite combinatorics plain and simple.
\newblock {\em J. Symb. Log.}, 83(3):1247--1281, 2018.

\bibitem{To:Partition}
S.~Todorcevic.
\newblock {\em Partition Problems in Topology}, volume~84 of {\em Contemporary
  mathematics}.
\newblock American Mathematical Society, Providence, Rhode Island, 1989.

\bibitem{vandersteeg2003models}
B.~van~der Steeg.
\newblock {\em Models in topology}.
\newblock Delft University Press, 2003.

\bibitem{vM:Introduction}
J.~van Mill.
\newblock An introduction to $\beta\omega$.
\newblock In K.~Kunen and J.~Vaughan, editors, {\em Handbook of Set-theoretic
  topology}, pages 503--560. North-Holland, 1984.

\bibitem{vanMill2001infinite}
J.~van Mill.
\newblock {\em The infinite-dimensional topology of function spaces}, volume~64
  of {\em North-Holland Mathematical Library}.
\newblock North--Holland, Amsterdam, 2001.

\bibitem{Ve:OCA}
B.~Velickovic.
\newblock {OCA} and automorphisms of {${\mathcal P}(\omega) /\Fin$}.
\newblock {\em Top. Appl.}, 49:1--13, 1992.

\bibitem{vignati2017nontrivial}
A.~Vignati.
\newblock Nontrivial homeomorphisms of \v{C}ech--{S}tone remainders.
\newblock {\em M\"unster J. Math.}, 10(1):189--200, 2017.

\bibitem{vignati2022rigidity}
A.~Vignati.
\newblock Rigidity conjectures for continuous quotients.
\newblock In {\em Ann. Sci. Ec. Norm. Super.}, volume~55, pages 1687--1738,
  2022.

\bibitem{wallman1938lattices}
H.~Wallman.
\newblock Lattices and topological spaces.
\newblock {\em Ann. Math.}, 39(1):112--126, 1938.

\bibitem{waraszkiewicz1934probleme}
Z.~Waraszkiewicz.
\newblock Sur un probl{\`e}me de {M}. {H}ahn.
\newblock {\em Fund. Math}, 22(1):180--205, 1934.

\bibitem{yu1991automorphism}
J.~Y.-C. Yu.
\newblock Automorphism in the \v{C}ech--{S}tone remainder of the reals.
\newblock Preprint, 1991.

\end{thebibliography}

\end{document}